\documentclass[a4paper,twoside,10pt,leqno]{article}
\usepackage{expl3}
\usepackage{amssymb,amsmath,amsthm}
\usepackage[all]{xy}
\usepackage[nottoc, notlof, notlot]{tocbibind}
\usepackage{stmaryrd}
\usepackage{tikz}
\usetikzlibrary{arrows}

\usepackage[english]{babel}
\IfFileExists{stix.sty}
{
    \usepackage{stix}
}
{
    \usepackage{mathrsfs}
    \usepackage{fouriernc}
    \newcommand{\nequiv}{\not\equiv}
}
\DeclareSymbolFont{calletters}{OMS}{cmsy}{m}{n}
\DeclareSymbolFontAlphabet{\mathcal}{calletters}
\IfFileExists{tgpagella.sty}
{
    \usepackage{tgpagella}
}{}
\usepackage[T1]{fontenc}
\usepackage[utf8x]{inputenc}

\makeatletter
\def\blfootnote{\gdef\@thefnmark{}\@footnotetext}
\makeatother

\title{Arc spaces, motivic measure and Lipschitz geometry of real algebraic sets}
\newcommand\shorttitle{Arc spaces, motivic measure and Lipschitz geometry}
\date{December 14, 2018}
\author{
\begin{tabular}[t]{c@{\extracolsep{8em}}c} 
\scshape Jean-Baptiste Campesato\thanks{Aix Marseille Univ, CNRS, Centrale Marseille, I2M, Marseille, France \newline \hspace*{1.8em} E-mail: \texttt{Jean-Baptiste.Campesato@univ-amu.fr} \newline \ }
&
\scshape Toshizumi Fukui\thanks{Department of Mathematics, Faculty of Science, Saitama University, \newline \hspace*{1.8em} 255 Shimo-Okubo, Sakura-ku, Saitama, 338-8570, Japan \newline \hspace*{1.8em} E-mail: \texttt{tfukui@rimath.saitama-u.ac.jp} \newline \ }
\\
&
\\
\scshape Krzysztof Kurdyka\thanks{Univ. Grenoble Alpes, Univ. Savoie Mont Blanc, CNRS, LAMA, 73000 Chamb\'ery, France \newline \hspace*{1.8em} E-mail: \texttt{Krzysztof.Kurdyka@univ-savoie.fr} \newline \ }~~\footnotemark[5]
&
\scshape Adam Parusi\'nski\thanks{Universit\'e C\^ote d'Azur, Universit\'e Nice Sophia Antipolis, CNRS, LJAD, France \newline \hspace*{1.8em} E-mail: \texttt{Adam.Parusinski@unice.fr} \newline \ }~~\thanks{Partially supported by ANR project LISA (ANR-17-CE40-0023-03).}
\end{tabular}
}

\usepackage{hyperref}
\usepackage{color}
\definecolor{darkred}{rgb}{.5,0,0}
\definecolor{darkgreen}{rgb}{0,.5,0}
\definecolor{darkblue}{rgb}{0,0,.5}
\hypersetup{
    pdfencoding=unicode,
    colorlinks=true,
    linkcolor=darkred,
    citecolor=darkgreen,
    urlcolor=darkblue,
    bookmarks=false,
    pdftoolbar=true,
    pdfmenubar=true,
    pdffitwindow=true,
    pdfstartview={FitH},
    pdfcreator={},
    pdfproducer={},
    pdfnewwindow=true,
}
\makeatletter
    \hypersetup{pdftitle={\@title},pdfauthor={J.-B. Campesato, T. Fukui, K. Kurdyka, A. Parusi\'nski}}
\makeatother

\usepackage[inner=4cm,outer=3cm,top=4cm,bottom=4cm,includeheadfoot,headsep=.5cm]{geometry}
\usepackage{fancyhdr}
\renewcommand{\thepage}{\arabic{page}}
\fancypagestyle{plain}{
\fancyhf{}

}

\theoremstyle{plain}
\newtheorem{thm}{Theorem}[section]
\newtheorem*{nthm}{Theorem}
\newtheorem{prop}[thm]{Proposition}
\newtheorem{cor}[thm]{Corollary}
\newtheorem{lemma}[thm]{Lemma}
\theoremstyle{definition}
\newtheorem{defn}[thm]{Definition}

\newtheorem{rem}[thm]{Remark}

\newtheorem{notation}[thm]{Notation}
\newtheorem{assumption}[thm]{Assumption}

\usepackage{enumitem}
\setlist[itemize]{labelindent=.6em, itemindent=1em, leftmargin=!, label=\textbullet}

\usepackage{ifmtarg}
\makeatletter
  \newcommand{\TODO}[1]{\@ifmtarg{#1}{\emph{\textbf{TODO}}~}{\emph{\textbf{TODO:}~#1~}}}
\makeatother

\newcommand{\AS}{\mathcal{AS}}
\renewcommand{\L}{\mathcal L}
\newcommand{\R}{\mathbb R}
\renewcommand{\P}{\mathbb P}
\newcommand{\N}{\mathbb N_{\ge0}}
\newcommand{\Np}{\mathbb N_{>0}}
\newcommand{\Ninfty}{\mathbb N_{\ge0}\cup\{\infty\}}
\newcommand{\sing}{\mathrm{sing}}
\newcommand{\ord}{\operatorname{ord}}
\newcommand{\LL}{\mathbb L}
\newcommand{\M}{\mathcal M}
\newcommand{\Z}{\mathbb Z}
\renewcommand{\d}{\mathrm d}
\newcommand{\Jac}{\operatorname{Jac}}
\newcommand{\jac}{\operatorname{jac}}
\newcommand{\detjac}{\operatorname{jac}}
\newcommand{\Reg}{\operatorname{Reg}}
\newcommand{\pd}[2]{\frac{\partial{#1}}{\partial{#2}}}
\newcommand{\J}{\mathcal J}
\renewcommand{\Im}{\operatorname{Im}}
\renewcommand{\div}{\operatorname{div}}
\newcommand{\rank}{\operatorname{rank}}
\newcommand{\Pol}{\mathcal{P}}
\newcommand{\Rf}{\mathcal{R}}

\makeatletter
\newcommand{\vast}{\bBigg@{4}}
\newcommand{\Vast}{\bBigg@{5}}
\makeatother

\newcommand{\quotient}[2]{\left.\raisebox{.05em}{$#1$}\middle/\raisebox{-.05em}{$#2$}\right.}

\usepackage{ifthen}
\newcommand{\clos}[2][]{\ensuremath{\overline{#2}\ifthenelse{\equal{#1}{}}{}{^{#1}}}}

\begin{document}
\maketitle

\begin{abstract}
We investigate connections between Lipschitz geometry of real algebraic varieties and properties of their arc spaces. For this purpose we develop motivic integration in the real algebraic set-up. We construct a motivic measure on the space of real analytic arcs. We use this measure to define a real motivic integral which admits a change of variables formula not only for the birational but also for generically one-to-one Nash maps.

As a consequence we obtain an inverse mapping theorem which holds for continuous rational maps and, more generally, for generically arc-analytic maps. These maps appeared recently in the classification of singularities of real analytic function germs.

Finally, as an application, we characterize in terms of the motivic measure, germs of arc-analytic homeomorphism between real algebraic varieties which are bi-Lipschitz for the inner metric.
\end{abstract}

\vfill

{\footnotesize
\noindent 2010 Mathematics Subject Classification. Primary: 14P99. Secondary: 26A16, 14E18, 14B05. \\
Keywords: Real algebraic geometry, motivic integration, singularities, Lipschitz maps, Nash functions, arc-analytic maps, arc-symmetric sets, virtual Poincaré polynomial.
}
\newpage

\tableofcontents

\section{Introduction}
In this paper we establish relations between the arc space and the Lipschitz geometry of a singular real algebraic variety.  

The interest in the Lipschitz geometry of real analytic and algebraic spaces emerged in the 70's of the last century by a conjecture of Siebenmann and Sullivan: there are only countably many local Lipschitz structures on real analytic spaces. Subsequently the Lipschitz geometry of real and complex algebraic singularities attracted much attention and various methods have been developed to study it: stratification theory \cite{Mos85,Par93}, $L$-regular decompositions \cite{Kur92,Par94-L,KP06,Paw08}, Lipschitz triangulations \cite{Val05}, non-archimedean geometry \cite{HY}, and recently, in the complex case, resolution and low dimensional topology \cite{BNP}. In the algebraic case Siebenmann and Sullivan's conjecture was proved in \cite{Par88}. The general analytic case was solved in \cite{Val08}.

In this paper we study various versions of Lipschitz inverse mapping theorems, with respect to the inner distance, for homeomorphisms $f:X\rightarrow Y$ between (possibly singular) real algebraic set germs. Recall that a connected real algebraic, and more generally a connected semialgebraic, subset $X\subset\R^N$ is path-connected (by rectifiable curves), so we have an \emph{inner} distance on $X$, defined by the infimum over the length of rectifiable curves joining two given points in $X$.

We assume that the homeomorphism $f$ is semialgebraic and generically arc-analytic. For instance the recently studied continuous rational maps \cite{Kuc09,KN,KKK} are of this type.

Arc-analytic mappings were introduced to real algebraic geometry in \cite{Kur88}. Those are the mappings sending by composition real analytic arcs to real analytic arcs. It was shown in \cite{BM90,Par94} that the semialgebraic arc-analytic mappings coincide with the blow-Nash mappings. Moreover, by \cite{PP}, real algebraic sets admit algebraic stratifications with local semialgebraic arc-analytic triviality along each stratum.

What we prove can be stated informally as follows: if $f^{-1}$ is Lipschitz, then so is $f$ itself. The problem is non-trivial even when the germs $(X,x)$ and $(Y,y)$ are non-singular \cite{FKP}. When these germs are singular, then the problem is much more delicate. In fact we have to assume that the motivic measures of the real analytic arcs drawn on $(X,x)$ and $(Y,y)$ are equal.

Developing a rigorous theory of motivic measure on the space of real analytic arcs for real algebraic sets is another main goal of this paper.

We state below a concise version of our main results. For more precise and more general statements see Theorems \ref{thm:IFT} and \ref{thm:mainLip}.

\begin{nthm}
Let $f:(X,x)\rightarrow(Y,y)$ be the germ of a semialgebraic generically arc-analytic homeomorphism between two real algebraic set germs, that are of pure dimension\footnote{For ease of reading, in the introduction we avoid varieties admitting points which have a structure of smooth submanifold of smaller dimension as in the handle of the Whitney umbrella $\{x^2=zy^2\}\subset\R^3$.} $d$.
Assume that the motivic measures of the real analytic arcs centered at $x$ in $X$ and of the real analytic arcs centered at $y$ in $Y$ are equal (see Section \ref{sec:motivic} for the definition of the motivic measure). Then
\begin{enumerate}
\item If the Jacobian determinant of $f$ is bounded from below then it is bounded from above and $f^{-1}$ is generically arc-analytic.
\item If the inverse $f^{-1}$ of $f$ is Lipschitz with respect to the inner distance then so is $f$.
\end{enumerate}
\end{nthm}

The proof of this theorem is based on motivic integration. Recall that in the case of complex algebraic varieties, motivic integration was introduced by M. Kontsevitch for non-singular varieties in order to avoid the use of $p$-adic integrals. Then the theory was developped and extended to the singular case in \cite{DL99,Bat98,DL02,Loo02}. The motivic measure is defined on the space of formal arcs drawn on an algebraic variety and takes values in a Grothendieck ring which encodes all the additive invariants of the underlying category. One main ingredient consists in reducing the study to truncated arcs in order to work with finite dimensional spaces. Notice that since the seminal paper of Nash \cite{Nash}, it has been established that the arc space of a variety encodes a great deal of information about its singularities.

In the real algebraic set-up, arguments coming from motivic integration were used in \cite{KP03,Fic05,Cam16,Cam17} to study and classify the singularities of real algebraic function germs.

In the present paper we construct a motivic measure and a motivic integral for possibly singular real algebraic varieties. Similarly to the complex case, the motivic integral comes together with a change of variables formula which is convenient to do actual computations in terms of resolution of singularities. In our real algebraic set-up this formula holds for generically one-to-one Nash maps and not merely for the birational ones.

A first difference of the present construction compared to the complex one, is that we work with real analytic arcs and not with all formal arcs. However, thanks to Artin approximation theorem, this difference is minor. More importantly, it is not possible to follow exactly the construction of the motivic measure in the complex case because of several additional difficulties arising from the absence in the real set-up of the Nullstellensatz and of the theorem of Chevalley (the image of a Zariski-constructible set by a regular mapping is Zariski-constructible).

The real motivic measure and the real motivic integral are constructed and studied in Section \ref{sec:motivic}. 

\section{Geometric framework}
Throughout this paper, we say that a subset $X\subset\R^N$ is an algebraic set if it is closed for the Zariski topology, i.e. $X$ may be described as the intersection of the zero sets of polynomials with real coefficients. We denote by $I(X)$ the ideal of $\R[x_1,\ldots,x_N]$ consisting of the polynomials vanishing on $X$. By noetherianity, we may always assume that the above intersection is indexed by a finite set\footnote{Actually, noticing that $f_1=\cdots=f_s=0\Leftrightarrow f_1^2+\cdots+f_s^2=0$, we may always describe a real algebraic set as the zero-set of only one polynomial.} and that $I(X)=(f_1,\ldots,f_s)$ is finitely generated. The dimension $\dim X$ of $X$ is the dimension of the ring $\Pol(X)=\quotient{\R[x_1,\ldots,x_N]}{I(X)}$ of polynomial functions on $X$.

The ring $\Rf(X)$ of regular functions on $X$ is given by the localization of $\Pol(X)$ with respect to the multiplicative set $\{h\in\Pol(X),\,h^{-1}(0)=\varnothing\}$. Regular maps are the morphisms of real algebraic sets.

Unless otherwise stated, we will always use the Euclidean topology and not the Zariski one (for instance for the notions of homeomorphism, map germ or closure).

We say that a $d$-dimensional algebraic set $X$ is non-singular at $x\in X$ if there exist $g_1,\ldots,g_{N-d}\in I(X)$ and an Euclidean open neighborhood $U$ of $x$ in $\R^N$ such that $U\cap X=U\cap V(g_1,\ldots,g_{N-d})$ and $\rank\left(\pd{g_i}{x_j}(x)\right)=N-d$. Then there exists an open semialgebraic neighborhood of $x$ in $V$ which is a $d$-dimensional Nash submanifold. Notice that the converse doesn't hold \cite[Example 3.3.12.b.]{BCR}. We denote by $\Reg(X)$ the set of non-singular points of $X$. We denote by $X_\sing=X\setminus\Reg(X)$ the set of singular points of $X$, it is an algebraic subset of strictly smaller dimension, see \cite[Proposition 3.3.14]{BCR}.

A semialgebraic subset of $\R^N$ is the projection of an algebraic subset of $\R^{N+m}$ for some $m\in\N$. Actually, by a result of Motzkin \cite{Mot70}, we may always assume that $m=1$. Equivalently, a subset $S\subset\R^N$ is semialgebraic if and only if there exist polynomials $f_i,g_{i,1},\ldots,g_{i,s_i}\in\R[x_1,\ldots,x_N]$ such that $$S=\bigcup_{i=1}^r\left\{x\in\R^N,\,f_i(x)=0,\,g_{i,1}(x)>0,\ldots,g_{i,s_i}(x)>0\right\}.$$
Notice that semialgebraic sets are closed under union, intersection and cartesian product. They are also closed under projection by the Tarski--Seidenberg Theorem. A function is semialgebraic if so is its graph.

We refer the reader to \cite{BCR} for more details on real algebraic geometry.

Let $X$ be a non-singular real algebraic set and $f:X\rightarrow\R$. We say that $f$ is a Nash function if it is $C^\infty$ and semialgebraic. Since a semialgebraic function satisfies a non-trivial polynomial equation and since a smooth function satisfying a non-trivial real analytic equation is real analytic \cite{Mal67,Sic70,Boc70}, we obtain that $f$ is Nash if and only if $f$ is real analytic and satisfies a non-trivial polynomial equation.

A subset of a real analytic variety is said to be arc-symmetric in the sense of \cite{Kur88} if, given a real analytic arc, either the arc is entirely included in the set or it meets the set at isolated points only. We are going to work with a slightly different notion defined in \cite{Par04}. We define $\AS^N$ as the boolean algebra generated by semialgebraic\footnote{A subset of $\P_\R^N$ is semialgebraic if it is for $\P_\R^N$ seen as an algebraic subset of some $\R^M$, or, equivalently, if the intersection of the set with each canonical affine chart is semialgebraic.} arc-symmetric subsets of $\P_\R^N$. We set $$\AS=\bigcup_{N\in\N}\AS^N.$$ 

Formally, a subset $A\subset\P_\R^N$ is an $\AS$-set if it is semialgebraic and if, given a real analytic arc $\gamma:(-1,1)\rightarrow\P_\R^N$ such that $\gamma(-1,0)\subset A$, there exists $\varepsilon>0$ such that $\gamma(0,\varepsilon)\subset A$.

Notice that closed $\AS$-subsets of $\P_\R^N$ are exactly the closed sets of a noetherian topology.

For more on arc-symmetric and $\AS$ sets we refer the reader to \cite{KP07}.


One important property of the $\AS$ sets that we rely on throughout this paper is that it admits an additive invariant richer than the Euler characteristic with compact support, namely the virtual Poincaré polynomial presented later in Section \ref{sec:GR-AS}. This is in contrast to the semialgebraic sets, for which, by a theorem of R. Quarez \cite{Qua01}, every additive homeomorphism invariant of semialgebraic sets factorises through the Euler characteristic with compact support.

Let $E,B,F$ be three $\AS$-sets. We say that $p:E\rightarrow B$ is an $\AS$ piecewise trivial fibration with fiber $F$ if there exists a finite partition $B=\sqcup B_i$ into $\AS$-sets such that $p^{-1}(B_i)\simeq B_i\times F$ where $\simeq$ means bijection with $\AS$-graph.

Notice that, thanks to the noetherianity of the $\AS$-topology, if $p:E\rightarrow B$ is locally trivial with fiber $F$ for the $\AS$-topology\footnote{i.e. for every $x\in B$ there is $U\subset B$ an $\AS$-open subset containing $x$ such that $p^{-1}(U)\simeq U\times F$.}, then it is an $\AS$ piecewise trivial fibration.

\section{Real motivic integration}\label{sec:motivic}
This section is devoted to the construction of a real motivic measure. Notice that a first step in this direction was done by R. Quarez in \cite{Qua01} using the Euler characteristic with compact support for semialgebraic sets. The measure constructed in this section takes advantage of the $\AS$-machinery in order to use the virtual Poincaré polynomial which is a real analogue of the Hodge--Deligne polynomial in real algebraic geometry. This additive invariant is richer than the Euler characteristic since it encodes, for example, the dimension.

Since real algebraic geometry is quite different from complex algebraic geometry as there is, for example, no Nullstellensatz or Chevalley's theorem, the classical construction of the motivic measure does not work as it is in this real context and it is necessary to carefully handle these differences.

\subsection{Real arcs and jets}
We follow the notations of \cite[\S2.4]{Cam16}.

\begin{defn}
The space of real analytic arcs on $\R^N$ is defined as $$\L(\R^N)=\left\{\gamma:(\R,0)\rightarrow\R^N,\,\gamma\text{ real analytic}\right\}$$
\end{defn}

\begin{defn}
For $n\in\N$, the space of $n$-jets on $\R^N$ is defined as $$\L_n(\R^N)=\quotient{\L(\R^N)}{\sim_n}$$ where $\gamma_1\sim_n\gamma_2\Leftrightarrow\gamma_1\equiv\gamma_2\mod t^{n+1}$.
\end{defn}

\begin{notation}
For $m>n$, we consider the following \emph{truncation maps}: $$\pi_n:\L(\R^N)\rightarrow\L_n(\R^N)$$ and $$\pi^m_n:\L_m(\R^N)\rightarrow\L_n(\R^N).$$
\end{notation}

\begin{defn}
For an algebraic set $X\subset\R^N$, we define the space of real analytic arcs on $X$ as $$\L(X)=\left\{\gamma\in\L(\R^N),\,\forall f\in I(X),\,f(\gamma(t))=0\right\}$$ and the space of $n$-jets on $X$ as $$\L_n(X)=\left\{\gamma\in\L_n(\R^N),\,\forall f\in I(X),\,f(\gamma(t))\equiv0\mod t^{n+1}\right\}.$$
The truncation maps induce the maps $$\pi_n:\L(X)\rightarrow\L_n(X)$$ and $$\pi^m_n:\L_m(X)\rightarrow\L_n(X).$$
\end{defn}

\begin{rem}
Notice that $\L_n(X)$ is a real algebraic variety. Indeed, let $f\in I(X)$ and $a_0,\ldots,a_n\in\R^N$, then we have the following expansion $$f(a_0+a_1t+\cdots+a_nt^n)=P_0^f(a_0,\ldots,a_n)+P_1^f(a_0,\ldots,a_n)t+\cdots+P_n^f(a_0,\ldots,a_n)t^n+\cdots$$ where the coefficients $P_i^f$ are polynomials. Hence $\L_n(X)$ is the algebraic subset of $\R^{N(n+1)}$ defined as the zero-set of the polynomials $P_i^f$ for $f\in I(X)$ and $i\in\{0,\ldots,n\}$. \\
In the same way, we may think of $\L(X)$ as an infinite-dimensional algebraic variety.
\end{rem}

\begin{rem}
When $X$ is non-singular the following equality holds: $$\L_n(X)=\pi_n(\L(X))$$
Indeed, using Hensel's lemma, we may always lift an $n$-jet to a formal arc on $X$ and then use Artin approximation theorem to find an analytic arc whose expansion coincides up to the degree $n+1$. However this equality doesn't hold anymore when $X$ is singular as highlighted in \cite[Example 2.30]{Cam16}. Hence it is necessary to distinguish the space $\L_n(X)$ of $n$-jets on $X$ and the space $\pi_n(\L(X))\subset\L_n(X)$ of $n$-jets on $X$ which may be lifted to real analytic arcs on $X$. We have the following exact statement.
\end{rem}

\begin{prop}[{\cite[Proposition 2.31]{Cam16}}]
Let $X$ be an algebraic subset of $\R^N$. Then the following are equivalent :
\begin{enumerate}[label=(\roman*),nosep]
\item $X$ is non-singular.
\item $\forall n\in\N$, $\pi_n:\L(X)\rightarrow\L_n(X)$ is surjective.
\item $\forall n\in\N$, $\pi^{n+1}_n:\L_{n+1}(X)\rightarrow\L_n(X)$ is surjective.
\end{enumerate}
\end{prop}

\begin{prop}[{\cite[Proposition 2.33]{Cam16}}]\label{prop:dimfibers}
Let $X$ be a $d$-dimensional algebraic subset of $\R^N$. Then
\begin{enumerate}[label=(\arabic*),ref=\ref{prop:dimfibers}.(\arabic*),nosep]
\item\label{item:truncfibers} For $m\ge n$, the dimensions of the fibers of ${\pi^{m}_n}_{|\pi_{m}(\L(X))}:\pi_m\left(\L(X)\right)\rightarrow\pi_n\left(\L(X)\right)$ are smaller than or equal to $(m-n)d$.
\item The fiber $\left(\pi^{n+1}_n\right)^{-1}(\gamma)$ of $\pi^{n+1}_n:\L_{n+1}(X)\rightarrow\L_n(X)$ is either empty or isomorphic to $T^{\mathrm{Zar}}_{\gamma(0)}X$.
\end{enumerate}
\end{prop}

\begin{thm}[A motivic corollary of Greenberg Theorem]\label{thm:greenberg}
Let $X\subset\R^N$ an algebraic subset. There exists $c\in\Np$ (depending only on $I(X)$) such that $$\forall n\in\N,\,\pi_n(\L(X))=\pi^{cn}_n(\L_{cn}(X))$$
\end{thm}
\begin{proof}
Assume that $I(X)=(f_1,\ldots,f_s)$.

By the main theorem of \cite{Gre66}, there exist $N\in\Np$, $l\in\Np$ and $\sigma\in\N$ (depending only on the ideal of $\R\{t\}[x_1,\ldots,x_N]$ generated by $f_i\in\R[x_1,\ldots,x_N]\subset\R\{t\}[x_1,\ldots,x_N]$) such that $\forall\nu\ge N,\,\forall\gamma\in\R\{t\}^N$, if $f_1(\gamma(t))\equiv\cdots\equiv f_s(\gamma(t))\equiv0\mod t^\nu$, then there exists $\eta\in\R\{t\}^N$ such that $\eta(t)\equiv\gamma(t)\mod t^{\left\lfloor\frac{\nu}{l}\right\rfloor-\sigma}$ and $f_1(\eta(t))=\cdots=f_s(\eta(t))=0$.

Fix $c=\max\left(l(\sigma+2),N\right)$. We are going to prove that $$\forall n\in\N,\,\pi_n(\L(X))=\pi^{cn}_n(\L_{cn}(X))$$

It is enough to prove that $\pi^{cn}_n(\L_{cn}(X))\subset\pi_n(\L(X))$ for $n\ge1$.

Let $n\ge1$. Let $\tilde\gamma\in\L_{cn}(X)$. Then there exists $\gamma\in\R\{t\}^N$ such that $\gamma(t)\equiv\tilde\gamma(t)\mod t^{cn+1}$ and $$f_1(\gamma(t))\equiv\cdots\equiv f_s(\gamma(t))\equiv0\mod t^{cn+1}$$

Notice that $cn+1\ge N$ so that there exists $\eta\in\R\{t\}^N$ such that $\eta(t)\equiv\gamma(t)\mod t^{\left\lfloor\frac{cn+1}{l}\right\rfloor-\sigma}$ and $f_1(\eta(t))=\cdots=f_s(\eta(t))=0$.

Since $$\left\lfloor\frac{cn+1}{l}\right\rfloor-\sigma>n$$ we have that $\pi^{cn}_n(\tilde\gamma)=\pi_n(\eta)\in\pi_n(\L(X))$.
\end{proof}

\begin{rem}
By Tarski--Seidenberg theorem, $\pi_n(\L(X))=\pi^{cn}_n(\L_{cn}(X))$ is semialgebraic as the projection of an algebraic set. However, $\pi_n(\L(X))$ may not be $\AS$ (and thus not Zariski-constructible) as shown in \cite[Example 2.32]{Cam16}.

This is a major difference with the complex case where $\pi_n(\L(X))$ is Zariski-constructible by Chevalley theorem as the projection of a complex algebraic variety.
\end{rem}

\begin{defn}
Let $X$ be an algebraic subset of $\R^N$. We define the ideal $H_X$ of $\R[x_1,\ldots,x_N]$ by $$H_X=\sum_{f_1,\ldots,f_{N-d}\in I(X)}\Delta(f_1,\ldots,f_{N-d})((f_1,\ldots,f_{N-d}):I(X))$$ where \begin{itemize}[nosep]
\item $d=\dim X$
\item $\Delta(f_1,\ldots,f_{N-d})$ is the ideal generated by the $N-d$ minors of the Jacobian matrix $$\left(\frac{\partial f_i}{\partial x_j}\right)_{\substack{i=1,\ldots,N-d\\j=1,\ldots,N}}$$
\item $((f_1,\ldots,f_{N-d}):I(X))=\left\{g\in\R[x_1,\ldots,x_N],\,gI(X)\subset(f_1,\ldots,f_{N-d})\right\}$ is the ideal quotient of the ideal $(f_1,\ldots,f_{N-d})$ by the ideal $I(X)$
\end{itemize}
\end{defn}

\begin{rem}
By \cite[Lemma 4.1]{Cam16}, $V(H_X)=X_\sing$.
\end{rem}

\begin{defn}
Let $X\subset\R^N$ be an algebraic subset and $e\in\N$. We set $$\L^{(e)}(X)=\left\{\gamma\in\L(X),\,\exists h\in H_X,\,h(\gamma(t))\nequiv0\mod t^{e+1}\right\}$$
\end{defn}

\begin{rem}\label{rem:singarcs}
From now on, we set $$\L(X_\sing)=\left\{\gamma\in\L(\R^N),\,\forall h\in H_X,\,h(\gamma(t))=0\right\}$$ and $$\L_n(X_\sing)=\left\{\gamma\in\L_n(\R^N),\,\forall h\in H_X,\,h(\gamma(t))\equiv0\mod t^{n+1}\right\}.$$
Notice that $$\left\{\gamma\in\L(\R^N),\,\forall h\in H_X,\,h(\gamma(t))=0\right\}=\left\{\gamma\in\L(\R^N),\,\forall f\in I(X_\sing),\,f(\gamma(t))=0\right\}$$ but be careful that
\begin{align*}
&\left\{\gamma\in\L_n(\R^N),\,\forall h\in H_X,\,h(\gamma(t))\equiv0\mod t^{n+1}\right\}\\
&\quad\quad\quad\quad\neq\left\{\gamma\in\L_n(\R^N),\,\forall f\in I(X_\sing),\,f(\gamma(t))\equiv0\mod t^{n+1}\right\}
\end{align*}
Notice also that since the proof of Greenberg Theorem \ref{thm:greenberg} is algebraic, it holds for $\L(X_\sing)$ (just use the ideal $H_X$ in the proof).
\end{rem}

\begin{rem}
$\L(X)=\left(\displaystyle\bigcup_{e\in\N}\L^{(e)}(X)\right)\bigsqcup\L(X_\sing)$
\end{rem}

The following proposition is a real version of \cite[Lemma 4.1]{DL99}. Its proof is quite similar to the one of \cite[Lemma 4.5]{Cam16}.
\begin{prop}\label{prop:ptf}
Let $X$ be a $d$-dimensional algebraic subset of $\R^N$ and $e\in\N$. Then, for $n\ge e$,
\begin{enumerate}[nosep,label=(\roman*)]
\item $\pi_n\left(\L^{(e)}(X)\right)\in\AS$
\item $\pi^{n+1}_n:\pi_{n+1}\left(\L^{(e)}(X)\right)\rightarrow\pi_{n}\left(\L^{(e)}(X)\right)$ is an $\AS$ piecewise trivial fibration with fiber $\R^d$.
\end{enumerate}
\end{prop}
\begin{proof}
By \cite[Lemma 4.7]{Cam16}, $\L^{(e)}(X)$ is covered by finitely many sets of the form $$A_{\mathbf f,h,\delta}=\left\{\gamma\in\L(\R^N),\,(h\delta)(\gamma(t))\nequiv0\mod t^{e+1}\right\}$$ where\ \ $\mathbf f=(f_1,\ldots,f_{N-d})\in I(X)^{N-d}$, $\delta$ is a $N-d$ minor of the Jacobian matrix $\left(\frac{\partial f_i}{\partial x_j}\right)_{\substack{i=1,\ldots,N-d\\j=1,\ldots,N}}$ and $h\in((f_1,\ldots,f_{N-d}):I(X))$.
Moreover, $$\L(X)\cap A_{\mathbf f,h,\delta}=\left\{\gamma\in\L(\R^N),\,f_1(\gamma(t))=\cdots=f_{N-d}(\gamma(t))=0,\,(h\delta)(\gamma(t))\nequiv0\mod t^{e+1}\right\},$$
so that $\displaystyle\L^{(e)}(X)=\L(X)\cap\bigcup_{\mathrm{finite}}A_{\mathbf f,h,\delta}=\bigcup_{\mathrm{finite}}\left(\L(X)\cap A_{\mathbf f,h,\delta}\right)$. \\
For $e'\le e$, we set $$A_{\mathbf f,h,\delta,e'}=\left\{\gamma\in A_{\mathbf f,h,\delta},\,\ord_t\delta(\gamma(t))=e',\,\ord_t\delta'(\gamma(t))\ge e',\,\text{for all $N-d$ minor $\delta'$ of $\left(\frac{\partial f_i}{\partial x_j}\right)$}\right\}$$
in order to refine the above cover: $\displaystyle\L^{(e)}(X)=\bigcup_{\mathrm{finite}}\left(\L(X)\cap A_{\mathbf f,h,\delta,e'}\right)$. \\

Fix some set $A=A_{\mathbf f,h,\delta,e'}\cap\L(X)$. Notice that if $\pi_n(\gamma)\in\pi_n(A)$ and if $\pi_{n+1}(\eta)\in\pi_{n+1}(\L^{(e)}(X))$ is in the preimage of $\pi_n(\gamma)$ by $\pi^{n+1}_n$ then $\pi_{n+1}(\eta)\in\pi_{n+1}(A)$. \\
Indeed, $\eta\in\L(X)$ so $f_1(\eta)=\cdots=f_{N-d}(\eta)=0$ and since $\pi_n(\eta)=\pi_n(\gamma)$, we also get that $(h\delta)(\eta(t))\nequiv0\mod t^{e+1}$, $\ord_t\delta(\eta(t))=e'$ and $\ord_t\delta'(\eta(t))\ge e'$. \\
Hence it is enough to prove the lemma for $\pi^{n+1}_n:\pi_{n+1}(A)\rightarrow\pi_n(A)$. \\

We are first going to prove that the fibers of  $\pi^{n+1}_n:\pi_{n+1}(A)\rightarrow\pi_n(A)$ are $d$-dimensional affine subspaces of $\R^N$. We can reorder the coordinates so that $\delta$ is the determinant of the first $N-d$ columns of $\Delta=\left(\frac{\partial f_i}{\partial x_j}\right)$. Then, similarly to the proof of \cite[Lemma 4.5]{Cam16}, there is a matrix $P$ such that $P\Delta=(\delta I_{N-d},W)$ and $\forall\gamma\in A,\,W(\gamma(t))\equiv0\mod t^{e'}$.

Fix $\gamma\in A$. The elements of the fiber of $\pi_{n+1}(A)\rightarrow\pi_n(A)$ over $\pi_{n}(\gamma)$, $\gamma\in A$, are exactly the $$\pi_{n+1}\big(\gamma(t)+t^{n+1}\nu(t)\big)$$ for $\nu\in\R\{t\}^d$ such that $\mathbf f(\gamma(t)+t^{n+1}\nu(t))=0$.

\noindent Using Taylor expansion, this last condition becomes $$\mathbf f(\gamma(t))+t^{n+1}\Delta(\gamma(t))\nu(t)+t^{2(n+1)}(\cdots)=0$$

\noindent Or equivalently, since $\gamma\in A$, $$t^{n+1}\Delta(\gamma(t))\nu(t)+t^{2(n+1)}(\cdots)=0$$

\noindent Multiplying by $t^{-n-1-e'}P$, we get
$$t^{-e'}\big(\delta(\gamma(t))I_{N-d},W(\gamma(t))\big)\nu(t)+t^{n+1-e'}(\cdots)=0$$

\noindent Notice that $\ord_t(\delta(\gamma(t))=e'$. Hence, by Hensel's lemma and Artin approximation theorem, the sought fiber is the set of $$\pi_{n+1}\big(\gamma(t)\big)+t^{n+1}\nu_0$$ with $\nu_0$ satisfying the linear system induced by
$$t^{-e'}\big(\delta(\gamma(t))I_{N-d},W(\gamma(t))\big)\nu_0\equiv0\mod t$$

Let $\nu_0$ be a solution, then its first $N-d$ coefficients are expressed as linear combinations of the last $d$. Therefore each fiber of $\pi^{n+1}_n:\pi_{n+1}(A)\rightarrow\pi_n(A)$ is a $d$-dimensional affine subspace of $\R^N$. \\

By Greenberg Theorem \ref{thm:greenberg}, there is a $c\in\N$ such that $\pi_{cn}(A)$ is an $\AS$-set. Then $\pi_n(A)$ is an $\AS$-set as the image of $\pi^{cn}_n:\pi_{cn}(A)\rightarrow\pi_n(A)$ whose fibers have odd Euler characteristic with compact support, see \cite[Theorem 4.3]{Par04}. \\

Finally, notice that $\pi_{n+1}(A)\subset\pi_{n}(A)\times\R^N$ and that $\pi^{n+1}_n:\pi_{n+1}(A)\rightarrow\pi_n(A)$ is simply the first projection. Then, according to the following lemma, $\pi^{n+1}_n:\pi_{n+1}(A)\rightarrow\pi_n(A)$ is an $\AS$ piecewise trivial fibration.
\end{proof}

\begin{lemma}
Let $A$ be an $\AS$-set, $\Omega\subset A\times\R^N$ be an $\AS$-set and $\pi:\Omega\rightarrow A$ be the natural projection. \\
Assume that for all $x\in A$, the fiber $\Omega_x=\pi^{-1}(x)$ is a $d$-dimensional affine subspace of $\R^N$. \\
Then $\pi:\Omega\rightarrow A$ is an $\AS$ piecewise trivial fibration.
\end{lemma}
\begin{proof}
Up to embedding the space of $d$-dimensional affine subspaces of $\R^N$ into the space of $d+1$-dimensional vector suspaces of $\R^{N+1}$, we may assume that the fibers are linear subspaces. \\

Denote by $G=\mathbb G_{N,d}$ the Grassmannian of $d$-dimensional linear subspaces of $\R^N$ and let $E\to G$ be the tautological bundle; i.e. for $g\in G$, the fiber $E_g$ is the subspace given by $g$. \\

We are first going to prove that the following set is $\AS$, $$\tilde A=\left\{(x,g)\in A\times G,\,\Omega_x=E_g\right\}.$$
Identifying $G$ with the set of symmetric idempotent $(N\times N)$-matrices of trace $d$, see \cite[Proof of Theorem 3.4.4]{BCR}, for $i=1,\ldots,N$ we define the regular map $w_i:G\to\R^N$ as the projection to the coordinates corresponding to the $i$th-column of such matrices. Then $E_g$ is linearly spanned by $\left(w_i(g)\right)$.
Hence $L_i=\left\{(v,g)\in\R^N\times G,\,v=w_i(g)\right\}$ is $\AS$. Thus
$$\left\{(x,v,g)\in A\times\R^N\times G,\,v=w_i(g)\in\Omega_x\right\}=(\Omega\times G)\cap(A\times L_i)$$
is $\AS$ and its projection $$X_i=\left\{(x,g)\in A\times G,\,w_i(g)\in\Omega_x\right\}$$ is also $\AS$ as the image of an $\AS$-set by an injective $\AS$-map, see \cite[Theorem 4.5]{Par04}. \\
Then $\tilde A = \bigcap_i X_i$ is $\AS$ as claimed.  \\

Let $x_0\in A$. Fix a coordinate system on $\R^N$ such that $\Omega_{x_0}=\left\{x_{d+1}=\cdots=x_N=0\right\}$ and fix the projection $\Lambda:\R^N\rightarrow\R^d$ defined by $\Lambda(x_1,\ldots,x_N)=(x_1,\ldots,x_d)$. 
Let $\omega:\tilde A\rightarrow\R^{N\choose d}$ be such that the coordinates of $\omega(x,g)$ are the $d$-minors of $\left(\Lambda(w_i(g))\right)_{i=1,\ldots,N}$. 
Then $$\tilde{A_0}=\left\{(x,g)\in\tilde A,\,\Lambda:\Omega_x\rightarrow\R^d \text{ is of rank }d\right\}$$ is an $\AS$-set as the complement of $\omega^{-1}(0)$. 
Therefore $$A_0=\left\{x\in A,\,\Lambda:\Omega_x\rightarrow\R^d \text{ is of rank }d\right\}$$ is $\AS$ as the image of the $\AS$-set $\tilde{A_0}$ by the projection to the first factor which is an injective $\AS$-map. \\
Thus $\Phi(x,v)=(x,\Lambda(v))$ is a bijection whose graph is $\AS$.
$$\xymatrix{\pi^{-1}(A_0) \ar[rr]^{\Phi} \ar[rd]_{\pi} & & A_0\times\R^d \ar[ld]^{\mathrm{pr}_{A_0}} \\ & A_0 & }$$
Consequently $\pi:\Omega\rightarrow A$ is locally trivial for the $\AS$-topology and hence it is an $\AS$ piecewise trivial fibration.
\end{proof}

\subsection{The Grothendieck ring of $\AS$-sets}\label{sec:GR-AS}
\begin{defn}
Let $K_0(\AS)$ be the free abelian group generated by $[X]$, $X\in\AS$, modulo
\begin{enumerate}[label=(\roman*),nosep]
\item Let $X,Y\in\AS$. If there is a bijection $X\rightarrow Y$ with $\AS$-graph, then $[X]=[Y]$;
\item If $Y\subset X$ are two $\AS$-sets, then $[X]=[X\setminus Y]+[Y]$.
\end{enumerate}
We put a ring structure on $K_0(\AS)$ by adding the following relation:
\begin{enumerate}[label=(\roman*),nosep,resume]
\item If $X,Y\in\AS$, then $[X\times Y]=[X][Y]$.
\end{enumerate}
\end{defn}

\begin{notation}
We set $0=[\varnothing]$, $1=[\mathrm{pt}]$ and $\LL=[\R]$.
\end{notation}

\begin{rem}
Notice that $0$ is the unit of the addition and $1$ the unit of the multiplication.
\end{rem}

\begin{rem}
If $p:E\rightarrow B$ is an $\AS$ piecewise trivial fibration with fiber $F$, then $$[E]=[B][F]$$
\end{rem}

\begin{defn}
We set $\M=K_0(\AS)\left[\LL^{-1}\right]$.
\end{defn}

The authors of \cite{MP03} proved there exists a unique additive (and multiplicative) invariant of real algebraic varieties up to biregular morphisms which coincides with the Poincaré polynomial for compact non-singular varieties. This construction relies on the weak factorization theorem. Then G. Fichou \cite{Fic05} extended this construction to $\AS$-sets up Nash isomorphisms.

Next, in \cite{MP11}, they gave a new construction of the virtual Poincaré polynomial, related to the weight filtration they introduced in real algebraic geometry. They proved it is an invariant of $\AS$-sets up to homeomorphism with $\AS$-graph. Actually, using the additivity, they proved it is an invariant of $\AS$-sets up to $\AS$-bijections (see \cite[Remark 4.15]{Cam17}).
\begin{thm}[{\cite{MP03,Fic05,MP11}}]
There is a unique ring morphism $\beta:K_0(\AS)\rightarrow\Z[u]$ such that if $X$ compact and non-singular then $$\beta([X])=\sum_{i\ge0}\dim H_i(X,\Z_2)u^i.$$
We say that $\beta([X])$ is the \emph{virtual Poincaré polynomial} of $X$. \\
Moreover, if $X\neq\varnothing$, $\deg\beta(X)=\dim X$ and the leading coefficient of $\beta(X)$ is positive.
\end{thm}

\begin{thm}[{\cite[Theorem 1.16]{Fic17}}]
The virtual Poincaré polynomial is a ring isomorphism $$\beta:K_0(\AS)\xrightarrow{\sim}\Z[u].$$
\end{thm}

\begin{rem}
The virtual Poincaré polynomial induces a ring isomorphism
$$\beta:\M\rightarrow\Z[u,u^{-1}].$$
\end{rem}

\begin{defn}
We define the ring $\widehat{\M}$ as the completion of $\M$ with respect to the ring filtration\footnote{i.e. $\mathcal F^{m+1}\M\subset\mathcal F^m\mathcal M$ and $\mathcal F^m\M\cdot\mathcal F^n\M\subset\mathcal F^{m+n}\M$. The last condition induces a ring structure on the group $\widehat{\M}$.} defined by the following subgroups induced by dimension $$\mathcal F^m\M=\left<[S]\LL^{-i}, i-\dim S \ge m\right>$$ i.e. $$\widehat{\M}=\quotient{\varprojlim\M}{\mathcal F^m\M}.$$
\end{defn}

\begin{prop}
The virtual Poincaré polynomial induces a ring isomorphism $$\beta:\widehat{\M}\rightarrow\Z[u]\llbracket u^{-1}\rrbracket.$$
\end{prop}
\begin{proof}
We have to prove that $$\varprojlim_m\quotient{\Z[u,u^{-1}]}{\mathcal F^m\Z[u,u^{-1}]}=\Z[u]\llbracket u^{-1}\rrbracket$$ where $\mathcal F^m\Z[u,u^{-1}]=\left<f\in\Z[u,u^{-1}], \deg f\le-m\right>$.

For $n<m$, we define $$\rho_{m,n}:\quotient{\Z[u,u^{-1}]}{\mathcal F^m\Z[u,u^{-1}]}\rightarrow\quotient{\Z[u,u^{-1}]}{\mathcal F^n\Z[u,u^{-1}]}$$ by $$\sum_{k=-m+1}^ra_ku^k\mapsto\sum_{k=-n+1}^ra_ku^k$$ and $$\rho_{m}:\Z[u]\llbracket u^{-1}\rrbracket\rightarrow\quotient{\Z[u,u^{-1}]}{\mathcal F^n\Z[u,u^{-1}]}$$ by $$\sum_{k=-\infty}^ra_ku^k\mapsto\sum_{k=-m+1}^ra_ku^k$$

By construction, $$\varprojlim_m\quotient{\Z[u,u^{-1}]}{\mathcal F^m\Z[u,u^{-1}]}=\left\{(f_m)\in\prod_{m\in\Z}\quotient{\Z[u,u^{-1}]}{\mathcal F^m\Z[u,u^{-1}]},\,n<m\Rightarrow\rho_{m,n}(f_m)=f_n\right\}$$

The morphism $$\varphi:\Z[u]\llbracket u^{-1}\rrbracket\rightarrow\varprojlim_m\quotient{\Z[u,u^{-1}]}{\mathcal F^m\Z[u,u^{-1}]}$$ defined by $f\mapsto(\rho_m(f))_{m\in\Z}$ is an isomorphism.
\end{proof}

\begin{defn}\label{defn:vdim}
For $\alpha\in\M$, we define the virtual dimension of $\alpha$ by $\dim\alpha=m$ where $m$ is the only integer such that $\alpha\in\mathcal F^{-m}\M\setminus\mathcal F^{-m+1}\M$.
\end{defn}

\begin{prop}
$\dim\alpha=\deg(\beta(\alpha))$
\end{prop}

\begin{rem}
Notice that for $x\in\M$, $\left(x+\mathcal F^m\M\right)_m$ defines a basis of open neighborhoods. This topology coincides with the one induced by the non-archimedean norm $\|\cdot\|:\M\rightarrow\R$ defined by $\|\alpha\|=e^{\dim(\alpha)}$. The completion $\widehat\M$ is exactly the topological completion with respect to this non-archimedean norm. Particularly,
\begin{itemize}[nosep]
\item Let $(\alpha_n)\in\M$, then $\alpha_n\rightarrow0$ in $\widehat\M$ if and only if $\dim(\alpha_n)\rightarrow-\infty$.
\item Let $(\alpha_n)\in\M$, then $\sum_n\alpha_n$ converges in $\widehat\M$ if and only if $\alpha_n\rightarrow0$ in $\widehat\M$.
\item The following equality holds in $\widehat\M$: $$(1-\LL^{-p})\sum_{i=0}^\infty\LL^{-pi}=1$$
\end{itemize}
\end{rem}

\begin{defn}\label{defn:order}
We define an order on $\widehat\M$ as follows. For $a,b\in\widehat\M$, we set $a\preceq b$ if and only if either $b=a$ or the leading coefficient of the virtual Poincaré polynomial $\beta(b-a)$ is positive.
\end{defn}

\begin{rem}
Notice that this real setting has good algebraic properties compared to its complex counterpart:
\begin{itemize}[nosep]
\item $K_0(\AS)$ is an integral domain whereas $K_0(\mathrm{Var}_\mathbb C)$ is not \cite{Poo02}. Indeed, there is no zero divisor in $K_0(\AS)$ whereas the class of the affine line is a zero divisor of $K_0(\mathrm{Var}_\mathbb C)$ \cite{Bor14} \cite{Mar16}. Notice that in particular $K_0(\mathrm{Var}_\mathbb C)\rightarrow\M_\mathbb C= K_0(\mathrm{Var}_\mathbb C)\left[\LL_\mathbb C^{-1}\right]$ is not injective.
\item The natural map $\M\rightarrow\widehat\M$ is injective. Indeed its kernel is $\cap_m\mathcal F^m\M$ and the virtual Poincaré polynomial allows us to conclude: if $\alpha\in\cap_m\mathcal F^m\M$, then, for all $m\in\Z$, $\deg\alpha\le-m$ and hence $\alpha=0$. In the complex case, it is not known whether $\M_\mathbb C\rightarrow\widehat\M_\mathbb C$ is injective.
\end{itemize}
\end{rem}

\subsection{Real motivic measure}
M. Kontsevitch introduced motivic integration in the non-singular case where the measurable sets were the cylinders by using the fact that they are stable. Still in the non-singular case, V. Batyrev \cite[\S6]{Bat98} enlarged the collection of measurable sets: a subset of the arc space is measurable if it may be approximated by stable sets.

Concerning the singular case, J. Denef and F. Loeser \cite{DL99} defined a measure and a first family of measurable sets including cylinders. Then, in \cite[Appendix]{DL02}, they used the tools they developped in the singular case to adapt the definition of V. Batyrev to the singular case. See also \cite{Loo02}.

From now on we assume that $X$ is a $d$-dimensional algebraic subset of $\R^N$.

\begin{defn}
A subset $A\subset\L(X)$ is said to be \emph{stable} at level $n$ if:
\begin{itemize}[nosep]
\item For $m\ge n$, $\pi_m(A)$ is an $\AS$-subset of $\L_m(X)$;
\item For $m\ge n$, $A=\pi_m^{-1}(\pi_m(A))$;
\item For $m\ge n$, $\pi^{m+1}_m:\pi_{m+1}(A)\rightarrow\pi_m(A)$ is an $\AS$ piecewise trivial fibration with fiber $\R^d$.
\end{itemize}
\end{defn}

\begin{rem}
Notice that, for the two first points, it is enough to verify that $\pi_n(A)\in\AS$ and that $A=\pi_n^{-1}(\pi_n(A))$ only for $n$. Indeed, then, for $m\ge n$, $\pi_m(A)=(\pi^m_n)^{-1}(\pi_n(A))$ is an $\AS$-set as inverse image of an $\AS$-set by a projection.
\end{rem}

Then the following proposition holds (notice that the condition $A=\pi_m^{-1}(\pi_m(A))$ is quite important).
\begin{prop}
If $A,B$ are stable subsets of $\L(X)$, then $A\cup B$, $A\cap B$ and $A\setminus B$ are stable too.
\end{prop}

\begin{rem}
Notice that $\L(X)$ may not be stable when $X$ is singular.
\end{rem}

\begin{defn}
For $A\subset\L(X)$ a stable set, we define its measure by $$\mu(A)=\frac{[\pi_n(A)]}{\LL^{(n+1)d}}\in\M,\,n\gg1.$$
\end{defn}

\begin{defn}
The virtual dimension of a stable set is $$\dim(A)=\dim(\pi_n(A))-(n+1)d,\,n\gg1.$$
\end{defn}

\begin{rem}
Notice that the previous definitions don't depend on $n$ for $n$ big enough.
\end{rem}

\begin{rem}
Notice that $\dim(A)=\dim(\mu(A))$ where the second dimension is the one introduced in Definition \ref{defn:vdim}.
\end{rem}

\begin{defn}\label{defn:measurable}
A subset $A\subset\L(X)$ is measurable if, for every $m\in\Z_{<0}$, there exist
\begin{itemize}[nosep]
\item a stable set $A_m\subset\L(X)$;
\item a sequence of stable sets $(C_{m,i}\subset\L(X))_{i\ge0}$
\end{itemize}
such that
\begin{itemize}[nosep]
\item $\forall i$, $\dim C_{m,i}<m$;
\item $A\Delta A_m\subset\cup C_{m,i}$
\end{itemize}
Then we define the measure of $A$ by $\displaystyle\mu(A)=\lim_{m\to-\infty}\mu(A_m)$.
\end{defn}

\begin{prop}\label{prop:Mlimit}
The previous limit is well defined in $\widehat\M$ and doesn't depend on the choices.
\end{prop}

The proof of the above Proposition, presented below, relies on the following two lemmas.

\begin{lemma}\label{lem:ASBaire}
Let $(A_i)_{i\in\N}$ be a decreasing sequence of non-empty $\AS$-sets $$A_1\supset A_2\supset\cdots$$
Then $$\bigcap_{i\in N}A_i\neq\varnothing.$$
\end{lemma}
\begin{proof}
Recall that $\clos[\AS]{A}$ denotes the smallest closed $\AS$-set containing $A$. We have the following sequence which stabilizes by noetherianity of the $\AS$-topology:
$$\clos[\AS]{A_1}\supset \clos[\AS]{A_2}\supset\cdots\supset\clos[\AS]{A_k}=\clos[\AS]{A_{k+1}}=\cdots$$

Recall that $\AS$-sets are exactly the constructible subsets of projective spaces for the $\AS$-topology whose closed sets are the semialgebraic arc-symmetric sets in the sense of \cite{Kur88}. Hence $A_l=\cup_{\mathrm{finite}}(U_i\cap V_i)$ where $U_i$ is $\AS$-open, $V_i$ is $\AS$-closed and $U_i\cap V_i\neq\varnothing$. We may assume that the $V_i$'s are irreducible (up to spliting them) and that $\clos[\AS]{U_i\cap V_i}=V_i$ (up to replacing $V_i$ by $\clos[\AS]{U_i\cap V_i}$). Hence we obtain the following decomposition as a union of finitely many irreducible closed subsets $\clos[\AS]{A_k}=\cup V_i$ (it is not necessarily the irreducible decomposition since we may have $V_i\subset V_j$).

Fix $Z$ an $\AS$-irreducible subset of $\clos[\AS]{A_k}$. By the previous discussion, for $l\ge k$, there exists $U_l$ an open dense $\AS$-subset of $Z$ such that $U_l\subset A_l$.

By \cite[Remark 2.7]{Par04}, $\dim(Z\setminus U_l)<\dim U_l$ so that $Z\setminus U_l$ is a closed subset of $Z$ with empty interior for the Euclidean topology. From Baire theorem, we deduce that the Euclidean interior of $\cup_{l\ge k}Z\setminus U_l$ is empty. Hence $\cap_{l\ge k}U_l$ is non-empty.
\end{proof}

The following lemma is an adaptation to the real context of \cite[Theorem 6.6]{Bat98}.
\begin{lemma}\label{lem:finitesubcov}
Let $A\subset\L(X)$ be a stable set and $(C_i)_{i\in\N}$ be a family of stable sets such that $$A\subset\bigcup_{i\in\N}C_i$$
Then there exists $l\in\N$ such that $$A\subset\bigcup_{i=0}^lC_i$$
\end{lemma}
\begin{proof}
Without loss of generality, we may assume that $C_i\subset A$ (up to replacing $C_i$ by $C_i\cap A$).

Set $D_i=A\setminus\left(C_1\cup\cdots\cup C_i\right)$ so that we get a decreasing sequence of stable sets $$D_1\supset D_2\supset D_3\supset\cdots$$ satisfying $$\bigcap_{i\in\N} D_i=\varnothing$$

Assume by contradiction that $A$ may not be covered by finitely many $C_i$, then $$\forall i\in\N, D_i\neq\varnothing$$

Now assume that $A$ is stable at level $n$ and that $D_i$ is stable at level $n_i\ge n$. Then $\pi_n(D_i)=\pi^{n_i}_n(\pi_{n_i}(D_i))\in\AS$ as the image of an $\AS$-set by a regular map whose fibers have odd Euler characteristic with compact support, see \cite[Theorem 4.3]{Par04}. Hence, by Lemma \ref{lem:ASBaire}, $$B_n=\bigcap_{i\in\N}\pi_n(D_i)\neq\varnothing$$
Choose $u_n\in B_n$.

Now set $$B_{n+1}=\bigcap_{i\in\N}\pi_{n+1}(D_i)\neq\varnothing$$ As before each $\pi_{n+1}(D_i)$ is a non-empty $\AS$-set. Notice that $(\pi^{n+1}_n)^{-1}(u_n)$ is a non-empty $\AS$-subset of $\L_{n+1}(X)$. Then, by Lemma \ref{lem:ASBaire}, $B_{n+1}\cap(\pi^{n+1}_n)^{-1}(u_n)\neq\varnothing$. This way, there exists $u_{n+1}\in B_{n+1}$ such that $\pi^{n+1}_{n}(u_{n+1})=u_n$.

Therefore, we may inductively construct a sequence $\left(u_m\in\L_m(X)\right)_{m\ge n}$ such that:
\begin{itemize}
\item $\displaystyle u_m\in B_m=\bigcap_{i\in\N}\pi_{m}(D_i)\neq\varnothing$;
\item $\pi^{m+1}_m(u_{m+1})=u_m$.
\end{itemize}

This defines an element $u\in\L(X)$ such that for all $m\ge n$, $\pi_m(u)\in B_m$. Hence for $i\in\N$, $\pi_{n_i}(u)\in B_{n_i}\subset\pi_{n_i}(D_i)$. Since $D_i$ is stable at level $n_i$, $u\in\pi_{n_i}^{-1}(\pi_{n_i}(D_i))=D_i$.

Therefore $u\in\bigcap D_i$ which is a contradiction.
\end{proof}

\begin{proof}[Proof of Proposition \ref{prop:Mlimit}]
We first prove that the limit is well defined. Let $A_m,C_{m,i}$ be as in the definition. Then for $m_1,m_2\in\Z_{<0}$, $$A_{m_1}\Delta A_{m_2}\subset\bigcup_{i\in\N}(C_{m_1,i}\cup C_{m_2,i})$$
By Lemma \ref{lem:finitesubcov}, there exists $l\in\N$ such that $$A_{m_1}\Delta A_{m_2}\subset\bigcup_{i=0}^l(C_{m_1,i}\cup C_{m_2,i})$$ hence $\dim(A_{m_1}\Delta A_{m_2})\le\max(m_1,m_2)$. Thus $\mu(A_m)$ is a Cauchy sequence and its limit is well defined in the completion $\widehat\M$.
\ \\

We now check that the limit doesn't depend on the choices. Let $A_m',C_{m,i}'$ be another choice of data for the measurability of $A$. Fix $m\in\Z_{<0}$ then $$A_m\Delta A_m'\subset\bigcup_{i\in\N}(C_{m,i}\cup C_{m,i}')$$
By Lemma \ref{lem:finitesubcov}, there exists $l\in\N$ such that $$A_m\Delta A_m'\subset\bigcup_{i=0}^l(C_{m,i}\cup C_{m,i}')$$

Hence $\dim(A_m\Delta A_m')<m$ and $\displaystyle\lim_{m\to-\infty}\mu(A_m)=\lim_{m\to-\infty}\mu(A_m')$.
\end{proof}

\begin{prop}
If $A,B$ are measurable subsets of $\L(X)$, then $A\cup B$, $A\cap B$ and $A\setminus B$ are measurable too.
\end{prop}
\begin{proof}
Assume that $A$ and $B$ are measurable, respectively with the data $A_m,C_{m,i}$ and $B_m,D_{m,i}$.
\begin{itemize}
\item $A\cup B$ is measurable since $$(A\cup B)\Delta(A_m\cup B_m)\subset\bigcup (C_{m,i}\cup D_{m,i})$$
\item In order to prove that $A\setminus B$ is measurable, we may use the previous point and assume that $B\subset A$ up to replacing $A$ by $A\cup B$. Similarly, we may assume that $B_m\subset A_m$. Then $$(A\setminus B)\Delta(A_m\setminus B_m)\subset\bigcup C_{m,i}\cup D_{m,i}$$
\item Using both previous points, we obtain that $A\cap B=(A\cup B)\setminus\left(((A\cup B)\setminus A)\cup((A\cup B)\setminus B)\right)$ is measurable.
\end{itemize}
\end{proof}

\begin{prop}
The measure is additive on disjoint unions: $$\mu(A\sqcup B)=\mu(A)+\mu(B)$$
\end{prop}
\begin{proof}
According to the previous proof we have
$$\mu(A\sqcup B)=\lim_{m\to\infty}\left(\mu(A_m)+\mu(B_m)-\mu(A_m\cap B_m)\right)$$
and $$0=\mu(A\cap B)=\lim_{m\to\infty}\mu(A_m\cap B_m)$$
Hence $$\mu(A\sqcup B)=\lim_{m\to\infty}\mu(A_m)+\lim_{m\to\infty}\mu(B_m)=\mu(A)+\mu(B)$$
\end{proof}

\begin{prop}\label{prop:measurableseries}
Let $(B_n)_{n\in\N}$ be a sequence of measurable sets with $\dim B_n\rightarrow-\infty$. \\
Then $B=\cup B_n$ is measurable and $$\mu(B)=\lim_{n\to+\infty}\mu\left(\bigcup_{k\le n}B_k\right).$$
Furthermore, if the sets $B_n$ are pairwise disjoint, then $$\mu(B)=\sum_{n=0}^{\infty}\mu\left(B_k\right).$$
\end{prop}
\begin{proof}
By Definition \ref{defn:measurable} for each $n\in \N$ and $m\in \Z_{< 0}$ there are stable sets $A_{n,m}$ and $C_{n,m,i}$, $\dim C_{n,m,i} < m$ such that $$B_n \Delta A_{n,m} \subset \bigcup_{i} C_{n,m,i}. $$
For $m\in \Z_{< 0}$ choose $N\in\N$ such that if $n\ge N$ then $\dim B_n <m$. \\
Note that then $\dim A_{n,m}<m$. Let us set $A_m=\displaystyle\bigcup_{k<N}A_{k,m}$. Then $$\bigcup_n B_n\Delta A_m\subset\bigcup_{n,i}C_{n,m,i}\cup\bigcup_{n\ge N}A_{n,m}.$$
This shows that $B$ is measurable. The other properties follows easily.
\end{proof}

\subsection{Measurability of the cylinders}
\begin{lemma}\label{lem:dimLS}
Let $X$ be a $d$-dimensional algebraic subset of $\R^N$. Let $S\subset X$ be an algebraic subset of $X$ with $\dim S<d$. For every $e\in\N$, there exists $N\in\N$ such that $$\forall i,n\in\N,\,n\ge i\ge N\Rightarrow \dim\left(\pi_n\left(\pi_i^{-1}\left(\L_i(S)\right)\right)\right)\le(n+1)d-e-1$$ where $\pi_n$ denotes the $n$-th truncation map for $X$ $$\forall n\in\N,\,\pi_n:\L(X)\rightarrow\L_n(X)$$ and where $\L(S)\subset\L(X)$ and $\forall i\in\N,\,\L_i(S)\subset\L_i(X)$.
\end{lemma}
\begin{proof}
By Greenberg Theorem \ref{thm:greenberg} applied to $S$, there exists $c\in\N$ such that $$\pi_e\left(\pi_{ce}^{-1}\left(\L_{ce}(S)\right)\right)=\pi_e\left(\L(S)\right)$$
Let $N=ce$ and $n\ge N$. By \ref{item:truncfibers} applied to $$\pi_n\left(\pi_{ce}^{-1}\left(\L_{ce}(S)\right)\right)\rightarrow\pi_e\left(\pi_{ce}^{-1}\left(\L_{ce}(S)\right)\right)$$ we get that $$\dim\left(\pi_n\left(\pi_{ce}^{-1}\left(\L_{ce}(S)\right)\right)\right)\le\dim\left(\pi_e\left(\pi_{ce}^{-1}\left(\L_{ce}(S)\right)\right)\right)+(n-e)d$$
But $$\pi_e\left(\pi_{ce}^{-1}\left(\L_{ce}(S)\right)\right)=\pi_e\left(\L(S)\right)$$ so that (see \cite[Proposition 2.33.(i)]{Cam16})$$\dim\left(\pi_n\left(\pi_{ce}^{-1}\left(\L_{ce}(S)\right)\right)\right)\le(e+1)(d-1)+(n-e)d=(n+1)-e-1$$
Now if $n\ge i\ge N(=ce)$, the result derives from the inclusion $$\pi_n\left(\pi_i^{-1}\left(\L_i(S)\right)\right)\subset\pi_n\left(\pi_{ce}^{-1}\left(\L_{ce}(S)\right)\right)$$
\end{proof}

\begin{defn}
Let $X\subset\R^N$ be an algebraic subset. For $i\in\Np$, we set $$C_i(X)=\L^{(i)}(X)\setminus\L^{(i-1)}(X).$$
\end{defn}

\begin{rem}
$\displaystyle C_i(X)=\left\{\gamma\in\L(X),\,\forall h\in H_X,\,\ord_th\circ\gamma\ge i,\,\exists\tilde h\in H_X,\,\ord_t\tilde h\circ\gamma=i\right\}$
\end{rem}

\begin{prop}\label{prop:dimCi}
For $i\in\Np$, $C_i(X)$ is stable and $$\lim_{i\to+\infty}\dim C_i(X)=-\infty$$
\end{prop}
\begin{proof}
Fix some $i\in\Np$. First, $C_i(X)$ is stable at level $i$ since the $\L^{(e)}(X)$ are stable by Proposition \ref{prop:ptf}.

Notice that $\pi_{i-1}(C_i(X))\subset\L_{i-1}(X_\sing)$. Hence $$C_i(X)\subset\pi_{i-1}^{-1}(\pi_{i-1}(C_i(X)))\subset\pi_{i-1}^{-1}\left(\L_{i-1}(X_\sing)\right)$$ and then $$\pi_i(C_i(X))\subset\pi_i\left(\pi_{i-1}^{-1}\left(\L_{i-1}(X_\sing)\right)\right).$$

As explained in Remark \ref{rem:singarcs}, we may apply Greenberg Theorem \ref{thm:greenberg} to $H_X$ so that Lemma \ref{lem:dimLS} holds for $X_\sing$.

Hence, for all $e\in\N$, there exists $N\in\N$ so that for $i\ge N$ we have $$\dim\left(\pi_i(C_i(X))\right)-(i+1)d\le\dim\left(\pi_i\left(\pi_{i-1}^{-1}\left(\L_{i-1}(X_\sing)\right)\right)\right)-(i+1)d\le -e$$
\end{proof}

\begin{cor}\label{cor:measALe}
A subset $A\subset\L(X)$ is measurable if and only if $\forall e\gg0,\, A\cap\L^{(e)}(X)$ 
is measurable.
\end{cor}
\begin{proof}
By Proposition \ref{prop:ptf} every $\L^{(e)}(X)$ is stable and therefore if $A$ is 
measurable so is every $A\cap\L^{(e)}(X)$. 

Suppose now that $\forall e\ge N,\, A\cap\L^{(e)}(X)$ is measurable. 
Then so are $A\cap C_i(X)$ for $i> N$. Hence $$A=\displaystyle\big(A\cap\L^{(N)}(X)\big)\cup\bigcup_{i>N}\big(A\cap C_i(X)\big)$$ is measurable by Proposition \ref{prop:measurableseries}. 
\end{proof}

\begin{defn}
A cylinder at level $n$ is a subset $A\subset\L(X)$ of the form $$A=\pi_n^{-1}(C)$$ for $C$ an $\AS$-subset of $\L_n(X)$. 
\end{defn}

\begin{rem}
A cylinder at level $n$ is a cylinder at level $m$ for $m\ge n$. Indeed $\pi_n=\pi^m_n\circ\pi_m$ so that $\pi_n^{-1}(C)=\pi_m^{-1}\left((\pi^m_n)^{-1}(C)\right)$ where $(\pi^m_n)^{-1}(C)\in\AS$ as the inverse image of an $\AS$-set by a projection.
\end{rem}

The following result derives from Proposition \ref{prop:ptf}.
\begin{prop}\label{prop:stabcyl}
If $X$ is non-singular, a cylinder of $\L(X)$ is stable.
\end{prop}

\begin{prop}
A cylinder $A\subset\L(X)$ is measurable and $$\mu(A)=\lim_{m\to+\infty}\mu\left(A\cap\L^{(m)}(X)\right)$$
\end{prop}
\begin{proof}
By Proposition \ref{prop:dimCi}, we may construct by induction an increasing map $\varphi:\Np\rightarrow\Np$ such that $$i\ge\varphi(s)\Rightarrow\dim C_i(X)<-s$$
Let $m\in\Z_{<0}$. Set $A_m=A\cap\L^{(\varphi(-m))}(X)$. Then $A_m$ is stable by Proposition \ref{prop:ptf} and $$A\Delta A_m=A\setminus\L^{(\varphi(-m))}(X)=A\cap\pi_{\varphi(-m)}^{-1}\left(\L_{\varphi(-m)}(X_\sing)\right)\subset\bigcup_{i\ge\varphi(-m)}C_i(X)$$ where $C_i(X)$ is stable with $\dim C_i(X)<m$. Hence $A$ is measurable and $$\mu(A)=\lim_{m\to+\infty}\mu\left(A\cap\L^{(\varphi(m))}(X)\right)$$

The second part of the statement derives from the fact that $\left(\mu\left(A\cap\L^{(m)}(X)\right)\right)_{m\in\Np}$ is already a Cauchy sequence. Assume that $A$ is a cylinder at level $s$ then $A\cap\L^{(m)}(X)$ is stable at level $\max(m,s)$. Indeed fix $k\in\N$. Then, for $n\ge m'\ge m\ge\max(\varphi(k),s)$, we get
\begin{align*}
\mu\left(A\cap\L^{(m)'}(X)\right)-\mu\left(A\cap\L^{(m)}(X)\right)&=\frac{\left[\pi_n\left(A\cap\L^{(m')}(X)\right)\right]}{\LL^{-(n+1)d}}-\frac{\left[\pi_n\left(A\cap\L^{(m)}(X)\right)\right]}{\LL^{-(n+1)d}}\\
&=\frac{\left[\pi_n\left(A\cap\L^{(m')}(X)\right)\setminus\pi_n\left(A\cap\L^{(m)}(X)\right)\right]}{\LL^{-(n+1)d}}\in\mathcal F^k\M
\end{align*}
\end{proof}

\begin{cor}
For $Y\subset X$ an algebraic subset, set $$\L(X,Y)=\left\{\gamma\in\L(X),\,\gamma(0)\in Y\right\}$$ then
\begin{itemize}[nosep]
\item $\L(X,Y)$ is a measurable subset of $\L(X)$;
\item in particular, $\L(X)$ is measurable.
\end{itemize}
\end{cor}
\begin{proof}
Indeed, $\L(X)=\pi_0^{-1}(X)$ and $\L(X,Y)=\pi_0^{-1}(Y)$ are cylinders.
\end{proof}

\begin{cor}
If $Y\subset X$ is an algebraic subset with $\dim Y<\dim X$, then $\L(Y)\subset\L(X)$ is measurable of measure $0$: $$\mu_{\L(X)}\left(\L(Y)\right)=0$$
\end{cor}
\begin{proof}
Notice that $\L(Y)$ is a countable intersection of cylinders: $$\L(Y)=\bigcap_{n\in\N}\pi_n^{-1}(\L_n(Y))$$

Then $\pi_n^{-1}(\L_n(Y))$ is measurable as a cylinder and $$\dim\mu\left(\pi_n^{-1}(\L_n(Y))\right)\le(n+1)(\dim Y-\dim X)\xrightarrow[n\to\infty]{}-\infty$$
\end{proof}

\subsection{Motivic integral and the change of variables formula}
\begin{defn}
Let $X\subset\R^N$ be an algebraic subset. Let $A\subset\L(X)$ be a measurable set. Let $\alpha:A\rightarrow\Ninfty$ be such that each fiber is measurable and $\mu(\alpha^{-1}(\infty))=0$. We say that $\LL^{-\alpha}$ is integrable if the following sequence converges in $\widehat\M$: $$\int_A\LL^{-\alpha}\d\mu=\sum_{n\ge0}\mu\left(\alpha^{-1}(n)\right)\LL^{-n}$$
\end{defn}

\begin{defn}\label{defn:gen1to1}
We say that a semialgebraic map $\sigma:M\rightarrow X$ between semialgebraic sets is \emph{generically one-to-one} if there exists a semialgebraic set $S\subset X$ satisfying $\dim(S)<\dim(X)$, $\dim\left(\sigma^{-1}(S)\right)<\dim(M)$ and $\forall p\in X\setminus S,\,\#\sigma^{-1}(p)=1$.
\end{defn}

\begin{defn}\label{defn:ordjac1}
Let $\sigma:M\rightarrow X$ be a Nash map between a $d$-dimensional non-singular algebraic set $M$ and an algebraic subset $X\subset\R^N$. For a real analytic arc $\gamma:(\R,0)\rightarrow M$, we set $$\ord_t\jac_\sigma(\gamma(t))=\min\left\{\ord_t\delta(\gamma(t)),\,\forall\delta\text{ $d$-minor of }\Jac_\sigma\right\},$$ where the Jacobian matrix $\Jac_\sigma$ is defined using a local system of coordinates around $\gamma(0)$ in $M$.
\end{defn}

The following lemma is a generalization of Denef--Loeser change of variables key lemma \cite[Lemma 3.4]{DL99} to generically one-to-one Nash maps in the real context.
\begin{lemma}[{\cite[Lemma 4.5]{Cam16}}]\label{lem:CoV}
Let $\sigma:M\rightarrow X$ be a proper generically one-to-one Nash map where $M$ is a non-singular $d$-dimensional algebraic subset of $\R^p$ and $X$ a $d$-dimensional algebraic subset of $\R^N$.
For $e,e'\in\N$ and $n\in\N$, set $$\Delta_{e,e'}=\left\{\gamma\in\L(M),\,\ord_t\jac_\sigma(\gamma(t))=e,\,\sigma_*(\gamma)\in\L^{(e')}(X)\right\}, \quad \Delta_{e,e',n}=\pi_n\left(\Delta_{e,e'}\right),$$
where $\sigma_*:\L(M)\rightarrow\L(X)$ is induced by $\sigma$. \\
Then for $n\ge\operatorname{max}(2e,e')$ the following holds:
\begin{enumerate}[label=(\roman*)]
\item Given $\gamma\in\Delta_{e,e'}$ and $\delta\in\L(X)$ with $\sigma_*(\gamma)\equiv\delta\mod t^{n+1}$ there exists a unique $\eta\in\L(M)$ such that $\sigma_*(\eta)=\delta$ and $\eta\equiv\gamma\mod t^{n-e+1}$.
\item Let $\gamma,\eta\in\L(M)$. If $\gamma\in\Delta_{e,e'}$ and $\sigma_*(\gamma)\equiv\sigma_*(\eta)\mod t^{n+1}$ then $\gamma\equiv\eta\mod t^{n-e+1}$ and $\eta\in\Delta_{e,e'}$.
\item The set $\Delta_{e,e',n}$ is a union of fibers of $\sigma_{*n}$.
\item $\sigma_{*n}(\Delta_{e,e',n})$ is an $\AS$-set and $\sigma_{*n|\Delta_{e,e',n}}:\Delta_{e,e',n}\rightarrow \sigma_{*n}(\Delta_{e,e',n})$ is an $\AS$ piecewise trivial fibration with fiber $\R^e$.
\end{enumerate}
\end{lemma}

\begin{lemma}\label{lem:invcyl}
Let $\sigma:X\rightarrow Y$ be a Nash map between algebraic sets. If $A\subset\L(Y)$ is a cylinder then $\sigma_*^{-1}(A)\subset\L(X)$ is also a cylinder.
\end{lemma}
\begin{proof}
Assume that $A=\pi_n^{-1}(C)$ where $C$ is an $\AS$-subset of $\L_n(Y)$. Then we have the following commutative diagram:
$$
\xymatrix{
\L(X) \ar[r]^{\sigma_*} \ar[d]_{\pi_n} & \L(Y) \ar[d]^{\pi_n} \\
\L_n(X) \ar[r]_{\sigma_{*n}} & \L_n(Y)
}
$$
Notice that $\sigma_{*n}$ is polynomial and thus its graph is $\AS$ so that the inverse image of an $\AS$-set by $\sigma_{*n}$ is also an $\AS$-set. Hence $\sigma_*^{-1}(A)=\pi_n^{-1}(\sigma_{*n}^{-1}(C))$ where $\sigma_{*n}^{-1}(C)$ is $\AS$.
\end{proof}

\begin{prop}\label{prop:preim}
Let $\sigma:M\rightarrow X$ be a proper generically one-to-one Nash map where $M$ is a non-singular $d$-dimensional algebraic subset of $\R^p$ and $X$ a $d$-dimensional algebraic subset of $\R^N$. \\
If $A\subset\L(X)$ is a measurable subset, then the inverse image $\sigma_*^{-1}(A)$ is also measurable.
\end{prop}
\begin{proof}
Let $$S'=\clos[\mathrm{Zar}]{\sigma^{-1}(X_\sing\cup S)\cup\Sigma_\sigma}$$ where $S\subset X$ is as in Definition \ref{defn:gen1to1} and $\Sigma_\sigma$ is the critical set of $\sigma$. Notice that the Zariski-closure of a semialgebraic set doesn't change its dimension. Therefore $\L(S')$ is a measurable subset of $\L(M)$ with measure $0$.

Hence $\sigma_*^{-1}(A)$ is measurable if and only if $\sigma_*^{-1}(A)\setminus\L(S')$ is measurable and then $$\mu\left(\sigma_*^{-1}(A)\right)=\mu\left(\sigma_*^{-1}(A)\setminus\L(S')\right)$$

Since $A$ is measurable, there exists $A_m$ and $C_{m,i}$ as in Definition \ref{defn:measurable}. Hence for all $m\in\Z_{<0}$, $$\sigma_*^{-1}(A)\Delta\sigma_*^{-1}(A_m)\subset\bigcup_i\sigma_*^{-1}(C_{m,i})$$ and 
\begin{equation}\label{eqn:meas}
\left(\sigma_*^{-1}(A)\setminus\L(S')\right)\Delta\left(\sigma_*^{-1}(A_m)\setminus\L(S')\right)\subset\bigcup_i\left(\sigma_*^{-1}(C_{m,i})\setminus\L(S')\right)
\end{equation}

By Lemma \ref{lem:invcyl} the sets $\sigma_*^{-1}(A_m)$ and $\sigma_*^{-1}(C_{m,i})$ are cylinders, therefore they are stable sets by Proposition \ref{prop:stabcyl} since $M$ is non-singular.

By definition of $S'$, $$\L(M)\setminus\L(S')\subset\bigcup_{e,e'}\Delta_{e,e'}$$

By Lemma \ref{lem:finitesubcov}, there exists $k$ such that $$\L(M)\setminus\L(S')\subset\bigcup_{e,e'\le k}\Delta_{e,e'}$$

Thus, by Lemma \ref{lem:CoV}, $\dim\left(\sigma_*^{-1}(C_{m,i})\setminus\L(S')\right)<k+m$.

This allows one to prove that $\sigma_*^{-1}(A)\setminus\L(S')$ is measurable by shifting the index $m$ in \eqref{eqn:meas}.
\end{proof}

\begin{prop}\label{prop:measIm}
Let $\sigma:M\rightarrow X$ be a proper generically one-to-one Nash map where $M$ is a non-singular $d$-dimensional algebraic subset of $\R^p$ and $X$ a $d$-dimensional algebraic subset of $\R^N$. \\
If $A\subset\L(M)$ is a measurable subset, then the image $\sigma_*(A)$ is also measurable.
\end{prop}

\begin{proof}
We use the same $S'$ as in the proof of Proposition \ref{prop:preim}. Then $\L(S')$ and $\sigma_*\left(\L(S')\right)$ have measure $0$ so that it is enough to prove that $\sigma_*\left(A\setminus\L(S')\right)$ is measurable.

\begin{lemma}\label{lem:stableIm}
There exists $k$ such that for every stable set $B\subset\L(M)\setminus\L(S')$, $\sigma_*(B)$ is stable and $\dim\left(\sigma_*(B)\right)<\dim(B)-k$.
\end{lemma}
\begin{proof}
By definition of $S'$ and Lemma \ref{lem:finitesubcov}, there exists $k$ such that $$B\subset\L(M)\setminus\L(S')\subset\bigcup_{e,e'\le k}\Delta_{e,e'}$$

Then the lemma derives from Lemma \ref{lem:CoV}.
\end{proof}

Assume that $A$ is measurable with the data $A_m,C_{m,i}$ then $$A\Delta A_m\subset\bigcup C_{m,i}$$ so that $$(A\setminus\L(S'))\Delta(A_m\setminus\L(S'))\subset\bigcup C_{m,i}\setminus\L(S')$$ and $$\sigma_*(A\setminus\L(S'))\Delta\sigma_*(A_m\setminus\L(S'))\subset\sigma_*\left((A\setminus\L(S'))\Delta(A_m\setminus\L(S'))\right)\subset\bigcup\sigma_*\left(C_{m,i}\setminus\L(S')\right)$$

Then we may conclude using Lemma \ref{lem:stableIm}.
\end{proof}

\begin{thm}\label{thm:CoV}
Let $\sigma:M\rightarrow X$ be a proper generically one-to-one Nash map where $M$ is a non-singular $d$-dimensional algebraic subset of $\R^p$ and $X$ a $d$-dimensional algebraic subset of $\R^N$. \\
Let $A\subset\L(X)$ be a measurable set. Let $\alpha:A\rightarrow\N\cup\{\infty\}$ be such that $\LL^{-\alpha}$ is integrable. \\
Then $\LL^{-(\alpha\circ\sigma_*+\ord_t\jac_\sigma)}$ is integrable on $\sigma_*^{-1}(A)$ and
$$\int_{A\cap\Im(\sigma_*)}\LL^{-\alpha}\d\mu_{\L(X)}=\int_{\sigma^{-1}_*(A)}\LL^{-(\alpha\circ\sigma_*+\ord_t\jac_\sigma)}\d\mu_{\L(M)}$$
where $\sigma_*:\L(M)\rightarrow\L(X)$ is induced by $\sigma$.
\end{thm}
\begin{proof}
Set $\beta=\alpha\circ\sigma_*+\ord_t\jac_\sigma$. By Proposition \ref{prop:preim}, $\sigma^{-1}_*(A)$ and the fibers of $\alpha\circ\sigma_*$ are measurable.

Notice that $$\beta^{-1}(n)=\bigsqcup_{e=0}^n\left((\alpha\circ\sigma_*)^{-1}(n-e)\cap(\ord_t\jac_\sigma)^{-1}(e)\cap\sigma_*^{-1}(A)\right)$$ so that the fibers of $\beta$ are measurable.

As in the proof of Proposition \ref{prop:preim}, up to replacing $\sigma_*^{-1}(A)$ by $\sigma_*^{-1}(A)\setminus\L(S')$, we may assume that $$\sigma_*^{-1}(A)\subset\bigcup_{e,e'\le k}\Delta_{e,e'}$$

Using Lemma \ref{lem:CoV}, we obtain
\begin{align*}
\int_{\sigma_*^{-1}(A)}\LL^{-(\alpha\circ\sigma_*+\ord_t\jac_\sigma)}\d\mu_{\L(M)} &= \sum_{e,e'\le k}\int_{\sigma_*^{-1}(A)\cap\Delta_{e,e'}}\LL^{-(\alpha\circ\sigma_*+\ord_t\jac_\sigma)}\d\mu_{\L(M)} \\
&=\sum_{e,e'\le k}\sum_{n\ge e}\mu\left(\gamma\in\sigma_*^{-1}(A)\cap\Delta_{e,e'},\,\alpha\circ\sigma_*(\gamma)=n-e\right)\LL^{-n} \\
&=\sum_{e,e'\le k}\sum_{n\ge e}\mu\left(\gamma\in A\cap\sigma_*(\Delta_{e,e'}),\,\alpha(\gamma)=n-e\right)\LL^{-(n-e)} \\
&=\sum_{e,e'\le k}\sum_{n\ge 0}\mu\left(\gamma\in A\cap\sigma_*(\Delta_{e,e'}),\,\alpha(\gamma)=n\right)\LL^{-n} \\
&=\sum_{n\ge 0}\sum_{e,e'\le k}\mu\left(\gamma\in A\cap\sigma_*(\Delta_{e,e'}),\,\alpha(\gamma)=n\right)\LL^{-n} \\
&=\sum_{n\ge 0}\mu\left(\gamma\in A\cap\Im(\sigma_*),\,\alpha(\gamma)=n\right)\LL^{-n} \\
&=\int_{A\cap\Im(\sigma_*)}\LL^{-\alpha}\d\mu_{\L(X)}
\end{align*}
Notice that $\Im(\sigma_*)$ is measurable by Proposition \ref{prop:measIm}.
\end{proof}

\section{An inverse mapping theorem for blow-Nash maps}
\subsection{Blow-Nash and generically arc-analytic maps}
\begin{defn}[{\cite[Définition 4.1]{Kur88}}]
Let $X$ and $Y$ be two real algebraic sets. We say that $f:X\rightarrow Y$ is arc-analytic if for every real analytic arc $\gamma:(-1,1)\rightarrow X$ the composition $f\circ\gamma:(-1,1)\rightarrow Y$ is also real analytic.
\end{defn}

\begin{defn}[{\cite[Definition 2.22]{Cam16}}]
Let $X$ and $Y$ be two algebraic sets. We say that the map $f:X\rightarrow Y$ is generically arc-analytic if there exists an algebraic subset $S\subset X$ satisfying $\dim S<\dim X$ and such that if $\gamma:(-1,1)\rightarrow X$ is a real analytic arc not entirely included in $S$, then the composition $f\circ\gamma:(-1,1)\rightarrow Y$ is also real analytic.
\end{defn}

\begin{defn}
Let $X$ and $Y$ be two algebraic sets. We say that $f:X\rightarrow Y$ is blow-Nash if $f$ is semialgebraic and if there exists a finite sequence of algebraic blowings-up with non-singular centers $\sigma:M\rightarrow X$ such that $f\circ\sigma:M\rightarrow Y$ is real analytic (and hence Nash).
\end{defn}

\begin{lemma}[{\cite[Lemma 2.27]{Cam16}}]\label{lem:BNisGenAA}
Let $f:X\rightarrow Y$ be a semialgebraic map between two real algebraic sets. Then $f:X\rightarrow Y$ is blow-Nash if and only if $f$ is generically arc-analytic. \\
\end{lemma}

\begin{rem}
In the non-singular case, the previous lemma derives from \cite{BM90} or \cite{Par94}.
\end{rem}

\begin{assumption}\label{hyp:S}
For the rest of this section we assume that $X\subset\R^N$ and $Y\subset\R^M$ are two $d$-dimensional algebraic sets and that $f:X\rightarrow Y$ is blow-Nash. Since $f$ is, in particular, semialgebraic, it is real analytic in the complement of an algebraic subset $S$ of $X$ of dimension $<d$. We may choose $S$ sufficiently big so that $S$ contains the singular set of $X$ and the non-analyticity set of $f$.
Because $f$ is blow-Nash we may suppose, moreover, that $f$ is analytic on every analytic arc $\gamma$ not included entirely in $S$. Then for every $\gamma\in\L(X)\setminus\L(S)$, $f\circ\gamma\in\L(Y)$.
\end{assumption}

We say that such $f$ \emph{is generically of maximal rank} if the Jacobian matrix of $f$ is of rank $d$ on a dense semialgebraic subset of $X\setminus S$. 

Let $\gamma\in\L(X)\setminus\L(S)$. Then the limit of tangent spaces $T_{\gamma(t)} X$ exists in the Grassmannian $\mathbb G_{N,d}$ of $d$-dimensional linear subspaces of $\R^N$. After a linear change of coordinates we may assume that this limit is equal to $\R^d\subset\R^N$. Then $(x_1,\ldots,x_d)$ is a local system of coordinates at every $\gamma(t)$, $t\ne 0$. Fix $J=\{j_1,\ldots,j_d)$ with $1\le j_1<\cdots<j_d\le M$. Then, for $t\ne 0$,
$$\d f_{j_1}\wedge\cdots\wedge\d f_{j_d}(\gamma(t))=\eta_J(t)\,\d x_{1}\wedge\cdots\wedge\d x_{d},$$
where $\eta_J(t)$ is a semialgebraic function, well-defined for $t\ne 0$.
Indeed, let $\Gamma_f\subset\R^{N+M}$ denote the graph of $f$ and let $\tau_{\Gamma_f}:\Reg(\Gamma_f)\to\mathbb G_{N+M,d}$ be the Gauss map. It is semialgebraic, see e.g. \cite[Proposition 3.4.7]{BCR}, \cite{Kur92}.
Denote by $\widetilde{\Gamma_f}$ the closure of its image and by $\pi_f:\widetilde{\Gamma_f}\to\Gamma_f$ the induced projection.
Then $\gamma$ lifts to a semialgebraic arc $\overline\gamma$ in $\widetilde{\Gamma_f}$. The limits $\lim_{t\to 0^+}\overline\gamma(t)$ and $\lim_{t\to 0^-}\overline\gamma(t)$ exist, and as follows from Proposition \ref{prop:orders} they coincide.

Denote by $E\to\mathbb G_{N+M,d}$ the tautological bundle. Thus each fiber of $E\to\mathbb G_{N+M,d}$ is a $d$-dimensional vector subspace of $\R^{N+M}$.
We denote by $(x_1,\ldots,x_N,f_1,\ldots f_M)$ the linear coordinates in $\R^{N+M}$.
Then the restriction of alternating $d$-forms to each $V^d\in\mathbb G_{N+M,d}$ gives an identity $$\d f_{j_1}\wedge\cdots\wedge\d f_{j_d} = \eta_J(V^d)\,\d x_{1}\wedge\cdots\wedge\d x_{d}$$ that defines a semialgebraic function $\eta_J(V_d)$ on $\mathbb G_{N+M,d}$ with values in $\R\cup\{\pm \infty\}$.
Then $\eta_J (t) = \eta_J (\overline \gamma(t))$.
As follows from Proposition \ref{prop:orders}, $\eta_J(t)$ is meromorphic and $\ord_t\eta_J\in \Z \cup \{\infty\}$.

The following notion generalizes the order defined in Definition \ref{defn:ordjac1}.

\begin{defn}\label{defn:ordjac}
\emph{The order of the Jacobian determinant of $f$ along $\gamma$} is defined as $$\ord_t\detjac_f(\gamma)=\min_J\{\ord_t\eta_J(t)\}.$$
If $\eta(t)\equiv0$ then we define its order as $+\infty$.
\end{defn}

\begin{defn}\label{defn:jacbounded}
We say that \emph{the Jacobian determinant of $f$ is bounded from above (resp. below)} if there exists $S\subset X$ as in \ref{hyp:S} such that for every $\gamma\in\L(X)\setminus\L(S)$, $\ord_t\detjac_f(\gamma)\ge0$ (resp. $\ord_t\detjac_f(\gamma)\le0$).  
\end{defn}

\subsection{Resolution diagram of $f$}
Let $g:M\rightarrow X$ be a Nash map where $M$ is a non-singular algebraic set and $X$ is an algebraic subset of $\mathbb R^N$. Denote by $\mathcal O_M$ the sheaf of Nash functions on $M$.  \\ Assume that $\dim M=\dim X=d$. Then the Jacobian sheaf $\J_g$ of $g$ is the sheaf of  $\mathcal O_M$-ideals generated, in a local system of coordinates $z_1,\ldots,z_d$ on $M$,  by 
$$
\J_g =
\left\langle\pd{\left(g_{i_1},\ldots,g_{i_d}\right)}{(z_1,\ldots,z_d)},\,
1\le i_1<\cdots<i_d\le N\right\rangle.
$$

Let $D=\cup D_i \subset M$ be a divisor with normal crossings.  We say that a local system of coordinates $z_1,\ldots,z_d$ at  $p\in M$ is compatible with $D$ if 
$D$ at $p$ is the zero set of a monomial in $z_i$ or $p\not \in D$.

\begin{prop}\label{prop:JsInvertible}
Let $g:M\rightarrow X$ be as in the previous definition. Then there exists $\sigma:\tilde M\rightarrow M$ the composition of a sequence of blowings-up with smooth algebraic centers and an algebraic divisor with simple normal crossings $D=\cup D_i\subset \tilde M$ such that in any local Nash system of coordinates compatible with $D$, 
$\J_{g\circ\sigma}$ is generated by a monomial. 
\end{prop}

\begin{proof}
First we fix a regular (in the algebraic sense) differential form $\omega_M$ of degree $d$ on $M$ which is not identically zero on every component of $M$.

There exists a sequence of blowings-up whose Jacobian determinant is a normal crossing divisor and such that the compositions with the coefficients of $\omega_M$ are also normal crossings, see for instance \cite[Theorem 1.10]{BM97}. Then the zero set of the pullback of $\omega_M$ is a divisor with simple normal crossings.
 
Up to composing with blowings-up, this allows us to assume that the zero set of $\omega_M$, denoted by $Z(\omega_M)$, is a divisor with simple normal crossings.

Since $M$, and hence $Z(\omega_M)$, is affine there is a regular function $\varphi$ on $M$ such that $Z(\omega_M)\subset \div \varphi$. By performing additional blowings-up we may assume that $\div (\varphi)$ is a divisor with normal crossings.  

For $I=\{i_1,\ldots,i_d\}\subset\{1,\ldots,N\}$, let $\pi_I:X\rightarrow\mathbb R^d$ be defined by $\pi_I(x_1,\ldots,x_N)=(x_{i_1},\ldots,x_{i_d})$. We consider the algebraic differential form $\omega_I=\pi^*(\d x_{i_1}\wedge\cdots\wedge\d x_{i_d})$. 
Then 
$$
\varphi g^*\omega_I=  h_I \omega_M,  
$$
where $h_I$ is a Nash function on $M$. By \cite[Proposition 2.11]{Cam17}, we may find a finite composition of blowings-up $\sigma:\tilde M\rightarrow M$, with smooth algebraic centers, such that $h_I\circ\sigma$ is locally a monomial times a Nash unit.  
More precisely, let $D \subset \tilde M$ be the union of $\sigma ^{-1} (\div \varphi)$ and the exceptional divisor of $\sigma$.  We may suppose that $D$ is with simple normal crossings and hence $h_I\circ\sigma$  equals a monomial times a Nash unit, in any local system of coordinates compatible with $D$.

Let $z_1,\ldots,z_d$ be such a local system of coordinates and let $\tilde g= g\circ \sigma$. Then
$$\tilde g^*\omega_I=\pd{\left(\tilde g_{i_1},\ldots,\tilde g_{i_d}\right)}{(z_1,\ldots,z_d)}\d z=\varphi^{-1}h_I\sigma^*\omega_M=z^{\alpha_I}u(z)\d z,$$
where $u$ is a unit. 

We may apply the above procedure to all $\omega_I$ and their differences.  
Then, by \cite[Beginning of the proof of Proposition 2.1]{Zar67}, see also \cite[Lemma 6.5]{BM90}, we conclude that the ideal generated by such 
$\tilde g^*\omega_I$ is, locally, principal and generated by a monomial.
\end{proof}

Let $p:\Gamma\rightarrow X$ be a composition of finitely many algebraic blowings-up such that $q=f\circ p:\Gamma\rightarrow Y$ is Nash and $\sigma:M\rightarrow\Gamma$ be an algebraic resolution of $\Gamma$ such that $\mathcal J_{p\circ\sigma}$ (resp. $\mathcal J_{q\circ\sigma}$) is locally generated by a monomial.
Notice that $M$ is a non-singular real algebraic variety and that $f\circ p\circ\sigma$ is Nash.
Note that if $M$ is not connected then $\mathcal J_{p\circ\sigma}$ can vanish identically on a connected component of $M$ if and only if $f$ is not generically of maximal rank.

We call $p:\Gamma\rightarrow X$ and $\sigma:M\rightarrow\Gamma$ satisfying the above properties \emph{a resolution diagram of $f$}. By Hironaka's desingularisation theorem and Proposition \ref{prop:JsInvertible}, such a diagram always exists but is not unique.

\begin{equation}\label{eqn:resolution}
\begin{gathered}
\xymatrix{
& M \ar[d]^\sigma & \\
& \Gamma \ar[dl]_p \ar[dr]^q & \\
X \ar[rr]_f & & Y
}
\end{gathered}
\end{equation}
 
By choosing the algebraic subset $S\subset X$ bigger (but still with $\dim S<d$) we may assume that $(p\circ\sigma)_*$ induces a bijection $\L(M)\setminus\L(S')\rightarrow\L(X)\setminus\L(S) $, where  $S'=(p\circ\sigma)^{-1}(S)$. Note that $\dim S'<d$.
Thus the diagram \eqref{eqn:resolution} induces a diagram
$$\xymatrix{
& \L(M)\setminus\L(S') \ar@{^{(}->>}[dl]_{(p\circ\sigma)_*} \ar[dr]^{(q\circ\sigma)_*} & \\
\L(X)\setminus\L(S) \ar[rr]_{f_*} & & \L(Y)
}$$
where we denote $f_*=(q\circ\sigma)_*\circ(p\circ \sigma)_*^{-1}$. 

Now we show how to compute the order of the Jacobian determinant of $f$ along $\gamma$  using a resolution diagram.

\begin{prop}\label{prop:orders}
Let $\gamma\in\L(X)\setminus\L(S)$ and let $\tilde\gamma=(p\circ \sigma)_*^{-1}(\gamma)$. Then
\begin{equation}
\ord_t\detjac_f(\gamma)=\ord_t\detjac_{q\circ\sigma}(\tilde\gamma(t))-\ord_t\detjac_{p\circ\sigma}(\tilde\gamma(t)).
\end{equation}
\end{prop}
\begin{proof}
The result derives from the chain rule which holds outside $S$.
\end{proof}

\begin{cor}\label{cor:jacbounded}
Suppose that $f$ is generically of maximal rank. Then the Jacobian determinant of $f$ is bounded from above, resp. from below, if and only if at every point of $M$ a local generator of $\J_{p\circ\sigma}$ divides a local generator of $\J_{q\circ\sigma}$, resp.  a local generator of $\J_{q\circ\sigma}$ divides a local generator of $\J_{p\circ\sigma}$.  
\end{cor}

\begin{rem}
We deduce from the previous corollary that if one of the conditions of Definition \ref{defn:jacbounded} is satisfied for one $S$, then it holds for every $S$.
\end{rem}

\subsection{An inverse mapping theorem}
\begin{thm}\label{thm:IFT}
Let $f:(X,x)\rightarrow(Y,y)$ be a germ of semialgebraic homeomorphism between real algebraic sets.
Assume that $\mu_{\L(X)}(\L(X,x))=\mu_{\L(Y)}(\L(Y,y))$. \\
If $f$ is generically arc-analytic and if the Jacobian determinant of $f$ is bounded from below, then the inverse map $f^{-1}:Y\rightarrow X$ is also generically arc-analytic and the Jacobian of $f$ is bounded from above. 
\end{thm}

\begin{rem}
Notice that arc-analyticity is an open condition for semialgebraic continuous maps (See \cite[Theorem 3.1]{KP05} where it is not necessary to assume that $f$ is bounded, up to composing $f$ with a real analytic diffeomorphism $\varphi:\R\rightarrow(-1,1)$). Hence, since the above statement is local, it is enough to use real analytic arcs centered at $x$ for the arc-analyticity condition.

The same holds for the boundedness of the Jacobian of $f$: we assume that the arcs of Definition \ref{defn:jacbounded} or Corollary \ref{cor:jacbounded} are centered at $x$.
\end{rem}

\begin{proof}[Proof of Theorem \ref{thm:IFT}]
We have the commutative diagram \eqref{eqn:resolution}. Notice that $E=(p\circ\sigma)^{-1}(0)$ is algebraic since $p\circ\sigma$ is regular. By Theorem \ref{thm:CoV},
\begin{align*}
\mu_{\L(X)}\left((p\circ\sigma)_*(\L(M,E))\right) &= \int_{(p\circ\sigma)_*(\L(M,E))}\LL^{-0}\d\mu_{\L(X)} \\
&= \int_{\L(M,E)}\LL^{-\ord_t\jac_{p\circ\sigma}}\d\mu_{\L(M)} \\
&= \sum_{n\ge0}\mu_{\L(M)}\left(\L(M,E)\cap\left(\ord_t\jac_{p\circ\sigma}\right)^{-1}(n)\right)\LL^{-n}
\end{align*}
Thus
\begin{align*}
\mu_{\L(X)}\left((p\circ\sigma)_*(\L(M,E))\right)\sum_{i\ge0}\LL^{-i} &=  \sum_{i\ge0}\sum_{n\ge0}\mu_{\L(M)}\left(\L(M,E)\cap\left(\ord_t\jac_{p\circ\sigma}\right)^{-1}(n)\right)\LL^{-(i+n)}\\
&=\sum_{n\ge0}\mu_{\L(M)}\left(\gamma\in\L(M,E),\,\ord_t\jac_{p\circ\sigma}(\gamma(t))\le n\right)\LL^{-n}
\end{align*}
Similarly
$$\mu_{\L(Y)}\left((q\circ\sigma)_*(\L(M,E))\right)\sum_{i\ge0}\LL^{-i}=\sum_{n\ge0}\mu_{\L(M)}\left(\gamma\in\L(M,E),\,\ord_t\jac_{q\circ\sigma}(\gamma(t))\le n\right)\LL^{-n}$$
Hence
\begin{align*}
&\left(\mu_{\L(Y)}\left((q\circ\sigma)_*(\L(M,E))\right)-\mu_{\L(X)}\left((p\circ\sigma)_*(\L(M,E))\right)\right)\sum_{i\ge0}\LL^{-i} \\
&\quad=\sum_{n\ge0}\Big(\mu_{\L(M)}\left(\gamma\in\L(M,E),\,\ord_t\jac_{q\circ\sigma}(\gamma(t))\le n\right)-\mu_{\L(M)}\left(\gamma\in\L(M,E),\,\ord_t\jac_{p\circ\sigma}(\gamma(t))\le n\right)\Big)\LL^{-n}
\end{align*}
Since we may lift a real analytic arc non-entirely included in the exceptional locus by $p\circ\sigma$, we have $$\mu_{\L(X)}\left((p\circ\sigma)_*(\L(M,E))\right)=\mu_{\L(X)}\left(\L(X,x)\right)$$ so that
\begin{align*}
&\left(\mu_{\L(Y)}\left((q\circ\sigma)_*(\L(M,E))\right)-\mu_{\L(X)}\left(\L(X,x)\right)\right)\sum_{i\ge0}\LL^{-i} \\
&\quad=\sum_{n\ge0}\Big(\mu_{\L(M)}\left(\gamma\in\L(M,E),\,\ord_t\jac_{q\circ\sigma}(\gamma(t))\le n\right)-\mu_{\L(M)}\left(\gamma\in\L(M,E),\,\ord_t\jac_{p\circ\sigma}(\gamma(t))\le n\right)\Big)\LL^{-n}
\end{align*}
Since $\mu_{\L(Y)}(\L(Y,y))=\mu_{\L(X)}(\L(X,x))$, we obtain
\begin{align*}
&\left(\mu_{\L(Y)}\left((q\circ\sigma)_*(\L(M,E))\right)-\mu_{\L(Y)}\left(\L(Y,y)\right)\right)\sum_{i\ge0}\LL^{-i} \\
&\quad=\sum_{n\ge0}\Big(\mu_{\L(M)}\left(\gamma\in\L(M,E),\,\ord_t\jac_{q\circ\sigma}(\gamma(t))\le n\right)-\mu_{\L(M)}\left(\gamma\in\L(M,E),\,\ord_t\jac_{p\circ\sigma}(\gamma(t))\le n\right)\Big)\LL^{-n}
\end{align*}
Since $M$ is non-singular, we may simply write
\begin{align*}
&\left(\mu_{\L(Y)}\left((q\circ\sigma)_*(\L(M,E))\right)-\mu_{\L(Y)}\left(\L(Y,y)\right)\right)\sum_{i\ge0}\LL^{-i} \\
&\quad=\sum_{n\ge0}\Big(\left[\gamma\in\L_n(M,E),\,\ord_t\jac_{q\circ\sigma}(\gamma(t))\le n\right]-\left[\gamma\in\L_n(M,E),\,\ord_t\jac_{p\circ\sigma}(\gamma(t))\le n\right]\Big)\LL^{-(n+2)d}
\end{align*}
Since the Jacobian determinant $f$ is bounded from below, using Proposition \ref{prop:orders}, we get that each summand of the RHS is positive or zero (in the sense of Definition \ref{defn:order}) because the leading coefficient of the virtual Poincaré polynomial of a non-empty $\AS$-set is positive:
\begin{align*}
&\left(\mu_{\L(Y)}\left((q\circ\sigma)_*(\L(M,E))\right)-\mu_{\L(Y)}\left(\L(Y,y)\right)\right)\sum_{i\ge0}\LL^{-i} \\
&\quad=\sum_{n\ge0}\Big(\left[\left\{\gamma\in\L_n(M,E),\,\ord_t\jac_{q\circ\sigma}(\gamma(t))\le n\right\}\setminus\left\{\gamma\in\L_n(M,E),\,\ord_t\jac_{p\circ\sigma}(\gamma(t))\le n\right\}\right]\Big)\LL^{-(n+2)d}
\end{align*}
Moreover, the LHS is negative or zero since $(q\circ\sigma)_*(\L(M,E))\subset\L(Y,y)$. \\

Assume that $f$ is not bounded from above, then at least one of the summand of the RHS is positive so that we obtain a contradiction. This proves that $f$ is bounded from above. \\

Furthermore, since the RHS is zero, we obtain that
\begin{equation}\label{eqn:Im}
\mu_{\L(Y)}\left((q\circ\sigma)_*(\L(M,E))\right)=\mu_{\L(Y)}\left(\L(Y,y)\right)
\end{equation}

We are now going to prove that $f^{-1}$ is generically arc-analytic so that it is blow-Nash.

Assume by contradiction there exists $\gamma\in\L(Y,y)$ not entirely included in $f(S)\cup Y_\sing$ which may not be lifted by $q\circ\sigma$. Nevertheless, by \cite[Proposition 2.21]{Cam16},
$$(q\circ\sigma)^{-1}(\gamma(t))=\sum_{i\ge0}c_it^{\frac i a},\,t\ge0$$
and
$$(q\circ\sigma)^{-1}(\gamma(t))=\sum_{i\ge0}d_i(-t)^{\frac i b},\,t\le0.$$

By assumption $(q\circ\sigma)^{-1}(\gamma(t))$ is not analytic so that either these expansions don't coincide or they have a non-integer exponent.
\begin{enumerate}
\item We first treat the latter case. Assume that $$(q\circ\sigma)^{-1}(\gamma(t))=\sum_{i=0}^mc_it^{i}+ct^{\frac{a}{b}}+\cdots,\,t\ge0,\,m<\frac a b<m+1,\,c\neq0.$$
Since $(q\circ\sigma)^{-1}:Y\setminus f(S)\rightarrow M$ is continuous and subanalytic, it is locally Hölder so that there exists $N\in\N$ satisfying for all real analytic arc $\eta(t)$ not entirely included in $f(S)\cup Y_\sing$, $$\eta(t)\equiv\gamma(t)\mod t^N\Rightarrow(q\circ\sigma)^{-1}(\eta(t))\equiv(q\circ\sigma)^{-1}(\gamma(t))\mod t^{m+1}.$$
Thus $\pi_N^{-1}\left(\pi_N(\gamma)\right)\subset\L(Y,y)\setminus(q\circ\sigma)_*(\L(M,E))$. \\
Notice that $\pi_N^{-1}\left(\pi_N(\gamma)\right)$ is measurable as a cylinder. Let $\rho:\tilde Y\rightarrow Y$ be a resolution of $Y$. Since $\gamma$ is not entirely included in the singular set of $Y$, there exists a unique real analytic arc $\tilde\gamma$ on $M$ such that $\gamma=\rho\circ\tilde\gamma$. Let $e=\ord_t\jac_\rho(\tilde\gamma(t))$ and $e'$ be such that $\gamma\in\L^{(e')}(Y)$. We may assume that $N\ge\max(e',2e)$. Then, by Lemma \ref{lem:CoV} and since $\tilde Y$ is non-singular,
\begin{align*}
\mu_{L(Y)}\left(\pi_N^{-1}\left(\pi_N(\gamma)\right)\right)&=\mu_{\L(\tilde Y)}\left(\pi_N^{-1}\left(\pi_N(\tilde\gamma)\right)\right)\LL^{-e} \\
&=\left[\pi_N\left(\pi_N^{-1}\left(\pi_N(\tilde\gamma)\right)\right)\right]\LL^{-(N+1)d-e} \\
&\neq 0
\end{align*}
Since $\pi_N^{-1}\left(\pi_N(\gamma)\right)\subset\L(Y,y)\setminus(q\circ\sigma)_*(\L(M,E))$, we obtain that $$\mu\left(\L(Y,y)\setminus(q\circ\sigma)_*(\L(M,E))\right)\neq0$$ which contradicts \eqref{eqn:Im}.
\item We now assume that
$$\tilde\gamma^+(t)=(q\circ\sigma)^{-1}(\gamma(t))=\sum_{i=0}^{m-1}c_it^{i}+ct^m+\cdots,\,t\ge0$$
and
$$\tilde\gamma^-(t)=(q\circ\sigma)^{-1}(\gamma(t))=\sum_{i=0}^{m-1}c_it^{i}+dt^m+\cdots,\,t\le0$$
with $c\neq d$.

Notice that $(q\circ\sigma)(\gamma^\pm(t))$ are analytic so that $\gamma(t)=(f\circ q\circ\sigma)(\gamma^+(t))=(f\circ q\circ\sigma)(\gamma^-(t))$. Since $f$ is a homeomorphism, we get $(q\circ\sigma)(\gamma^+(t))=(q\circ\sigma)(\gamma^-(t))$. Since this real analytic arc is not entirely included in $S$, it may be uniquely lifted by $q\circ\sigma$ so that $\gamma^+(t)=\gamma^-(t)$. Hence $c=d$ and we obtain a contradiction.
\end{enumerate}

Thus, for all $\gamma\in\L(Y,y)\setminus\L(f(S)\cup Y_\sing)$ there exists $\tilde\gamma\in\L(M,E)$ such that $(q\circ\sigma)(\tilde\gamma(t))=\gamma(t)$. Then $f^{-1}(\gamma(t))=(p\circ\sigma)(\tilde\gamma(t))$ which is real analytic. Therefore $f^{-1}$ is generically arc-analytic and so blow-Nash.
\end{proof}

\begin{rem}
Notice that, in the above proof, we do not need a homeomorphism $f:X\rightarrow Y$ but only a homeomorphism of $f:\overline{\Reg(X)}\rightarrow\overline{\Reg(Y)}$.
\end{rem}

Under the assumptions of the previous theorem, we derive the following corollary from Lemma \ref{lem:BNisGenAA}.
\begin{cor}\label{cor:IFTcor1}
Let $f:(X,x)\rightarrow(Y,y)$ be a semialgebraic homeomorphism germ between real algebraic sets with $\dim X=\dim Y$.
Assume moreover that $\mu_{\L(X)}(\L(X,x))=\mu_{\L(Y)}(\L(Y,y))$. \\
If $f$ is blow-Nash and if the Jacobian determinant of $f$ is bounded from below, then the inverse $f^{-1}$ is also blow-Nash and the Jacobian determinant of $f$ is bounded from above. 
\end{cor}

\begin{rem}
Notice that in the previous results we don't assume that $X=Y$ contrary to \cite[Main Theorem 3.5]{Cam16}.
\end{rem}

\begin{thm}\label{thm:CompareMeasures}
Let $f:(X,x)\rightarrow(Y,y)$ be a semialgebraic homeomorphism germ between algebraic sets with $\dim X=\dim Y$.
If $f$ is generically arc-analytic and if the Jacobian determinant of $f$ is bounded from below, then $\mu_{\L(X)}(\L(X,x))\preceq\mu_{\L(Y)}(\L(Y,y))$.
\end{thm}
\begin{proof}
Following the beginning of the proof of Theorem \ref{thm:IFT}, we obtain :
\begin{align*}
&\left(\mu_{\L(Y)}\left(\L(Y,y)\right)-\mu_{\L(X)}\left(\L(X,x)\right)\right)\sum_{i\ge0}\LL^{-i} \\
&\quad\succeq\left(\mu_{\L(Y)}\left((q\circ\sigma)_*(\L(M,E))\right)-\mu_{\L(X)}\left(\L(X,x)\right)\right)\sum_{i\ge0}\LL^{-i} \\
&\quad=\sum_{n\ge0}\Big(\mu_{\L(M)}\left(\gamma\in\L(M,E),\,\ord_t\jac_{q\circ\sigma}(\gamma(t))\le n\right)-\mu_{\L(M)}\left(\gamma\in\L(M,E),\,\ord_t\jac_{p\circ\sigma}(\gamma(t))\le n\right)\Big)\LL^{-n} \\
&\quad\succeq0
\end{align*}
\end{proof}

\section{An inverse mapping theorem for inner-Lipschitz maps}

\subsection{Inner distance}

Let $X$ be a connected semialgebraic subset of $\R^N$ equipped with the standard Euclidean distance.
We denote by $d_X$ the {\it inner} (also called \emph{geodesic}) distance in $X$. By definition, for $p,q\in X$, the inner distance $d_X(p,q)$ is the infimum over the length of all
rectifiable curves joining $p$ to $q$ in $X$. By \cite{KO97}, $d_X(p,q)$ is the infimum over the length of continuous semialgebraic curves in $X$ joining $p$ and $q$. 
It is proven in \cite{KO97} that $d_X$ can be approximated uniformly by subanalytic distances.

We recall some results from \cite{KO97}, based on \cite{Kur92}. Let $\varepsilon>0$, we say that a connected semialgebraic set $\Gamma\subset\R^N$ is \emph{$K_\varepsilon$-regular} if for any $p,q\in\Gamma$ we have $$d_\Gamma(p,q)\le(1+\varepsilon)|p-q|.$$
We state now a semialgebraic version of \cite[Proposition 3]{KO97}.

\begin{prop}\label{prop:ko}
Let $X\subset\R^N$ be a semialgebraic set and $\varepsilon>0$. Then there exists a finite decomposition 
$X = \bigcup_{\nu \in V} \Gamma_\nu$  such that:
\begin{enumerate}
\item each $ \Gamma_\nu$ is a semialgebraic connected analytic submanifold of  
$\R^N$,
\item each $\Gamma_\nu$ is $K_\varepsilon$-regular. 
\end{enumerate}
\end{prop}

\begin{rem}
Given a finite family of semialgebraic sets $X_i,\,i\in I$, we can find a decomposition satisfying the above conditions and such that for any $i\in I$, $\nu\in V$, we have: either $\Gamma_\nu\subset X_i$ or $\Gamma_\nu\cap X_i=\emptyset$.
\end{rem}

For a $C^1$ map $f:X '\rightarrow\R^{M}$ defined on a submanifold $X'$ of $\R^N$ we denote by $
D_pf:T_pX\rightarrow\R^{M}
$
its differential at $p\in X'$.  Then the norm of $D_p f$ is defined by $$\Vert D_pf\Vert=\sup\left\{|D_pf(v)|:\,v\in T_p,\,|v|=1\right\}.$$

\begin{lemma}\label{lem:Lip}
Assume that $f_\nu:\Gamma_\nu\rightarrow\R^{M}$ is a $C^1$-map, such that for any $p\in\Gamma_\nu$ we have $\Vert D_pf_\nu\Vert\le L$. Then $f_\nu$ is $(1+\varepsilon)L$-Lipschitz with respect to the Euclidean distance, hence it extends continuously on $\overline\Gamma_\nu$ to a Lipschitz map with the same constant.
\end{lemma}

\begin{proof}
Let $p,q\in\Gamma_\nu$ and $\varepsilon'>\varepsilon$, then, by \cite{KO97}, there exists a $C^1$-semialgebraic arc $\lambda:[0,1]\rightarrow\Gamma_\nu$ such that $p=\lambda(0)$, $q=\lambda(1)$ of the length $|\lambda|\le(1+\varepsilon')|p-q|$.
It follows that
$$
|f_\nu(p)-f_\nu(q)|\le L|\lambda|\le(1+\varepsilon')L|p-q|.
$$
We obtain the conclusion passing to the limit $\varepsilon'\rightarrow\varepsilon$.

Notice that, on any metric space, a Lipschitz mapping extends continuously to the closure with the same Lipschitz constant.
\end{proof}

Let $X$ and $Y$ be locally closed connected semialgebraic subsets respectively of $\R^N$ and $\R^{M}$.
They are equipped with the inner distances $d_X$ and $d_Y$, respectively. Let 
$$f:X\rightarrow Y$$ be a continuous semialgebraic map.
Then there exists a semialgebraic set $X'\subset X$, which is open and dense in $X$,  such that the connected components of $X'$ are analytic submanifolds of $\R^N$, possibly of different dimensions.
Moreover $f$ restricted to each connected component of $X'$ is analytic.

\begin{prop}\label{prop:inlip}
The following conditions are equivalent:
\begin{enumerate}[label=(\roman*),ref=\ref{prop:inlip}.(\roman*)]
\item $ d_Y (f(p),f(q))\le Ld_X(p,q)$ for any $p,q\in X$, \label{item:metric}
\item $\Vert D_pf\Vert \le L$ for any $p\in X'$. \label{item:norm}
\end{enumerate}
\end{prop}

\begin{proof}
The implication $\ref{item:metric}\Rightarrow\ref{item:norm}$ is obvious since 
at a smooth point $p\in X$, the inner and Euclidean distances are asymptotically equal.

To prove the converse let us fix $p,q\in X$. For any $\varepsilon>0$ there exists a continuous semialgebraic arc $\lambda:[0,1]\rightarrow X$ such that $p=\lambda(0)$, $q=\lambda(1)$ of the length $|\lambda|\le(1+\varepsilon)d_X(p,q)$. By Proposition \ref{prop:ko} there exists a finite decomposition $X'=\bigcup_{\nu\in V}\Gamma_\nu$ into $K_\varepsilon$-regular semialgebraic connected analytic submanifolds of $\R^N$. Let $X''=\bigcup_{\nu \in V'}\Gamma_\nu$ be the union of those $\Gamma_\nu$ which are open in $X'$.
Note that $X''$ is dense in $X'$. It follows that $X\subset\bigcup_{\nu \in V'}\overline{\Gamma_\nu}$. Since the arc $\lambda$ is semialgebraic there exists a finite sequence $0=t_0<\cdots<t _k=1$ such that each $\lambda([t_i,t_{i+1}])\subset\overline{\Gamma_\nu}$ for some $\nu\in V'$.
By Lemma \ref{lem:Lip} the length of $f(\lambda([t_i, t_{i+1}]))$ is bounded by  $(1+\varepsilon)|\lambda([t_i,t_{i+1}])|$. Hence 
$$|f(\lambda([0, 1]))| = \sum_{i=0}^{k-1}|f(\lambda([t_i,t_{i+1}])|\le(1+\varepsilon)L\sum_{i=0}^{k-1}|\lambda([t_i,t_{i+1}])|\le(1+\varepsilon)L|\lambda |. $$
 
Thus $$d_Y(f(p),f(q))\le|f(\lambda([0,1]))|\le(1+\varepsilon)L|\lambda|\le(1+\varepsilon)^2Ld_X(p,q)$$
We conclude by taking the limit as $\varepsilon\rightarrow0$.
\end{proof}

\subsection{An inverse mapping theorem}
We  suppose now that $f:X\rightarrow Y$ satisfies Assumption \ref{hyp:S}. Thus it is a blow-Nash map between two real algebraic sets of dimension $d$. Let $\gamma\in\L(X)\setminus\L(S)$. Let us adapt the notation introduced in the paragraph after Assumption \ref{hyp:S}. In particular we assume that the limit of tangent spaces $T_{\gamma(t)} X$ in the Grassmannian $\mathbb G_{N,d}$ is equal to $\R^d\subset\R^N$.  Then, for every $i=1,\ldots,d$ and every $j=1,\ldots,M$
$$\eta_{i,j}(t)=\frac{\partial f_{j}}{\partial x_i}$$ 
is semialgebraic. Thus the order of $\eta_{i,j}(t)$, as $t\rightarrow 0^+$ is a well defined rational number (or $+\infty$ if $f_j$ vanishes identically on $\gamma$).

\begin{defn}
\emph{The order of the Jacobian matrix of $f$ along $\gamma$} is defined as
$$\ord_{t\to 0^+}\Jac_f(\gamma(t))=\min_{i,j}\{\ord_{t\to 0^+}\eta_{i,j}(t)\}.$$
\end{defn}

\begin{rem}
The above notion shouldn't be confused with the order of the Jacobian \emph{determinant} defined in Definition \ref{defn:ordjac}. 
\end{rem}

\begin{rem}
It is likely that $\eta_{i,j}(t)$ is actually meromorphic and it is not necessary, in the above definition, to restrict to $t\to 0^+$. We leave it as an open problem.
\end{rem}

\begin{defn} 
We say that \emph{the Jacobian matrix of $f$ is bounded from above} if there is an $S$ such that for every $\gamma\in\L(X)\setminus\L(S)$, $\ord_{t\to 0^+} \Jac_f(\gamma(t))\ge 0$.
\end{defn}

One may show again that if the above condition is satisfied for one $S$ they are satisfied for every $S$. 

The following result follows from Proposition \ref{prop:inlip}.

\begin{prop}\label{prop:CforIL}
Let $f:(X,x)\rightarrow(Y,y)$ be a semialgebraic homeomorphism germ between two real algebraic set germs with $\dim(X,x)=\dim(Y,y)$. Then $f:\overline{\Reg(X)}\rightarrow\overline{\Reg(Y)}$ is inner Lipschitz iff the Jacobian matrix of $f$ is bounded from above.
\end{prop}

\begin{thm}\label{thm:mainLip}
Let $f:(X,x)\rightarrow(Y,y)$ be a semialgebraic homeomorphism germ between two real algebraic set germs with $\dim (X,x)=\dim (Y,y)$.
Assume that $\mu_{\L(X)}(\L(X,x))=\mu_{\L(Y)}(\L(Y,y))$.
If $f$ is generically arc-analytic and $f^{-1}:\overline{\Reg(Y)}\rightarrow\overline{\Reg(X)}$ is inner Lipschitz, then $f^{-1}:Y\rightarrow X$ is also generically arc-analytic and $f:\overline{\Reg(X)}\rightarrow\overline{\Reg(Y)}$ is inner Lipschitz.
\end{thm}

\begin{rem}
Notice that both previous results involve the closure of the regular parts of the algebraic sets. The obtained sets $\overline{\Reg(X)}$ and $\overline{\Reg(Y)}$ do not contain any part of smaller dimension but they still may not be smooth submanifolds. \\
For instance, for the Whitney umbrella $X=\{x^2=zy^2\}$, $\overline{\Reg(X)}$ consists in the canopy (i.e. the $z\ge0$ part of $X$). Therefore $\overline{\Reg(X)}$ is singular along the half-axis $\{(0,0,z),\,z\ge0\}$. However it doesn't contain the handle of the Whitney umbrella (i.e. $\{(0,0,z),\,z<0\}$) which is a smooth manifold of dimension $1$ whereas $\dim X=2$.
\end{rem}

\begin{proof}[Proof of Theorem \ref{thm:mainLip}.]
To simplify the exposition we suppose that $X$, and hence $Y$ as well, is pure-dimensional.  
That is $X = \overline{ \Reg(X)}$ and $Y = \overline{\Reg(Y)}$.  The proof in the general case is similar.

First we apply Proposition \ref{prop:inlip} to $f^{-1}$. Hence the Jacobian determinant of $f^{-1}$ is bounded from above. Therefore the Jacobian determinant of $f$ is bounded from below and we can apply to $f$ Theorem \ref{thm:IFT}. This shows that $f^{-1}$ is generically arc-analytic and that the Jacobian determinant of $f$ is bounded from above and below.

Now we show that the Jacobian matrix of $f$ is bounded from above. 
Let $\gamma\in\L(X)\setminus\L(S)$. We may assume, as explained above, that $\R^d\subset\R^N$ is the limit of tangent spaces $T_{\gamma(t)} X$.  
Similarly by considering the limit of $T_{f(\gamma(t))} Y$ we may assume that it equals $\R^d\subset\R^M$.  
Then $y_1,\ldots,y_d$ form a local system of coordinates on $Y$ at every $f(\gamma(t))$, $t\ne 0$.  
By the assumptions the matrix $(\pd{x_i}{y_j})(f(\gamma(t))$ is bounded and its determinant is a unit.  Therefore, by the cofactor formula, its inverse, that is $(\pd{x_i}{y_j})(f(\gamma(t))$ is bounded. This shows that $f$ is inner Lipschitz by Proposition \ref{prop:CforIL}. 
\end{proof}

\begin{rem}
Notice that, in the above proof, we do not need a homeomorphism $f:X\rightarrow Y$ but only a homeomorphism of $f:\overline{\Reg(X)}\rightarrow\overline{\Reg(Y)}$.
\end{rem}

\end{document}